\ifx\justbeingincluded\undefined
\documentclass{svproc}
\usepackage[utf8]{inputenc}
\usepackage{latexsym}
\usepackage{amsmath}
\usepackage{amssymb}
\usepackage{graphicx}
\usepackage{hyperref}
\usepackage{url}

\usepackage{booktabs}
\usepackage{esvect}
\usepackage{cleveref}
\crefname{equation}{}{}

\newcommand{\vecg}[1]{\boldsymbol{#1}}    
\newcommand{\dd}{\hspace{0.1cm} \mathrm{d}}     
\newcommand{\tsum}{{\textstyle\sum\limits}}     
\newcommand{\norm}[1]{\| #1 \|}                 
\newcommand{\abs}[1]{| #1 |}                    

\newcommand{\RR}{\mathbb{R}}
\newcommand{\NN}{\mathbb N}
\newcommand{\parameterDomain}{\Xi}
\newcommand{\parameter}[1]{y_{#1}}
\newcommand{\parameterVector}{y}
\newcommand{\iPC}{q}
\newcommand{\nPC}{Q}

\newcommand{\nFE}{N}
\newcommand{\iKL}{m}
\newcommand{\nKL}{M}
\newcommand{\spatialDomain}{D}
\newcommand{\spatialDomainBoundary}{\partial D}
\newcommand{\stochasticDomain}{\Omega}
\newcommand{\stochasticDomainValue}{\omega}
\newcommand{\sigmaAlgebra}{\mathcal F}
\newcommand{\probabilityMeasure}{\mathbb P}
\newcommand{\pTuple}{(\stochasticDomain,\sigmaAlgebra,\probabilityMeasure)}
\newcommand{\probabilityDensity}{\rho}
\newcommand{\stochasticDegree}{k}

\begin{document}
\mainmatter              
\fi

\title{A Bramble-Pasciak Conjugate Gradient Method for Discrete Stokes Problems with Lognormal Random Viscosity}
\titlerunning{BPCG for Stokes with lognormal viscosity}  
%
\author{Christopher M\"uller\inst{1,2} \and Sebastian Ullmann\inst{1,2} \and
Jens Lang\inst{1,2}}
\authorrunning{Christopher M\"uller et al.} 
%
%
\index{M\"ueller, C.}
\index{Ullmann, S.}
\index{Lang, J.}
\institute{Graduate School of Computational Engineering, Technische Universit\"at Darmstadt, 
  Dolivostr. 15, 64293, Darmstadt, Germany\\
\email{cmueller@gsc.tu-darmstadt.de},\\ WWW home page:
\texttt{http://www.graduate-school-ce.de/index.php?id=680}
\and
Department of Mathematics, Technische Universit\"at Darmstadt, 
  Dolivostr. 15, 64293, Darmstadt, Germany}

\maketitle              

\begin{abstract}
  We study linear systems of equations arising from a stochastic Galerkin finite element discretization of saddle point problems with random data and its iterative solution. We consider the Stokes flow model with random viscosity described by the exponential of a correlated random process and shortly discuss the discretization framework and the representation of the emerging matrix equation. Due to the high dimensionality and the coupling of the associated symmetric, indefinite, linear system, we resort to iterative solvers and problem-specific preconditioners. As a standard iterative solver for this problem class, we consider the block diagonal preconditioned MINRES method and further introduce the Bramble-Pasciak conjugate gradient method as a promising alternative. This special conjugate gradient method is formulated in a non-standard inner product with a block triangular preconditioner. From a structural point of view, such a block triangular preconditioner enables a better approximation of the original problem than the block diagonal one. We derive eigenvalue estimates to assess the convergence behavior of the two solvers with respect to relevant physical and numerical parameters and verify our findings by the help of a numerical test case. We model Stokes flow in a cavity driven by a moving lid and describe the viscosity by the exponential of a truncated Karhunen-Lo\`eve expansion. Regarding iteration counts, the Bramble-Pasciak conjugate gradient method with block triangular preconditioner is superior to the MINRES method with block diagonal preconditioner in the considered example.
\end{abstract}
\section{Introduction}
We study the stochastic Galerkin finite element (SGFE) method 
\cite{BabushkaEtAl2004,GhanemSpanos1991,GunzburgerEtAl2014} as a tool to approximate statistical quantities in the context 
of saddle point problems with random input data. Stochastic Galerkin (SG) methods rely on a representation of the 
random input based on a vector of random variables with known probability density. The starting point of the 
method is a weak formulation not only over the spatial domain but also over the image domain of this 
random vector. The SGFE approach enables the computation of the unknown solution coefficients via a 
Galerkin projection onto the finite dimensional tensor product space of a finite element (FE) 
space for the spatial dependencies and a global polynomial/SG space for the random vector. 

Concerning convergence rates, SG methods are often superior to more robust stochastic 
methods -- such as Monte-Carlo sampling -- when the input data exhibits certain regularity structures 
\cite{Bachmayr2017,HoangSchwab2014}. This is a property shared by the main competitors 
of the SG approaches: Stochastic collocation methods \cite{BabuskaEtAl2007,GunzburgerEtAl2014} 
similarly exploit the structure of the random input and further rely on uncoupled solutions of the 
underlying deterministic problem, just like sampling methods. An advantage of SG methods is that rigorous 
error analysis can be used to analyze them, but they are more challenging from a computational point of view: 
A block structured system of coupled deterministic problems must be solved. However, this can be done efficiently 
using iterative methods and problem-specific preconditioners, see e.g. 
\cite{ErnstEtAl2009,MuellerEtAl2017,PowellElman2009,PowellUllmann2010,EUllmannEtAl2012}.
 .

We build on these results and consider the Bramble-Pasciak conjugate gradient (BPCG) method 
\cite{BramblePasciak1988} in the SGFE setting with lognormal data. We compare it to the MINRES approach, the standard Krylov 
subspace solver for the problem class we consider, and investigate the performance of both solvers with 
respect to different problem and discretization parameters. This is done both analytically 
and based on a numerical test case.

As a saddle point problem, we consider the Stokes flow model in a bounded domain 
$\spatialDomain \subset \RR^2$ with boundary $\spatialDomainBoundary$. The vector 
$x = (x_1,x_2)^{\text{T}} \in \spatialDomain$ denotes the spatial coordinates. The Stokes equations are a simplification 
of the Navier-Stokes equations and describe the behavior of a velocity 
field $u = (u_1,u_2)^{\text{T}}$ and a pressure field $p$ subject to viscous and external forcing. We also 
introduce the probability space $\pTuple$, where $\stochasticDomain$ denotes the set of elementary events, 
$\sigmaAlgebra$ is a $\sigma$-algebra on $\stochasticDomain$ and 
$\probabilityMeasure:\sigmaAlgebra \rightarrow [0,1]$ is a probability measure. 

Input data are inherently uncertain due to either a lack of knowledge or simply imprecise measurements. 
Taking into account this variability in the model, we assume that the viscosity is a random field
$\nu = \nu({x},\stochasticDomainValue): \spatialDomain \times \stochasticDomain \to \RR$.
Since the input uncertainty propagates through the model, the solution components also have to be considered 
as random fields. In summary, the strong form of the Stokes equations with uncertain viscosity is the 
following: 

Find ${u}= {u}({x},\stochasticDomainValue)$ and $p=p({x},\stochasticDomainValue)$ 
such that, $\probabilityMeasure$-almost surely,
\begin{alignat}{2}
-\nabla\cdot \big(\nu({x},\stochasticDomainValue) \,  \nabla  {u}({x},\stochasticDomainValue)\big) + 
\nabla p({x},\stochasticDomainValue)  &= {f}({x}) &\quad& \mathrm{in} \hspace{0.2cm} \spatialDomain 
\times\stochasticDomain, \notag\\
\label{eq:continuous_stoch_problem}
\nabla \cdot {u}({x},\stochasticDomainValue)   &= 0 &\quad& \mathrm{in} \hspace{0.2cm} \spatialDomain
\times \stochasticDomain, \\
{u}({x},\stochasticDomainValue)  &= {g}({x}) &\quad& \mathrm{on} \hspace{0.2cm} 
\spatialDomainBoundary \times \stochasticDomain.\notag 
\end{alignat}
Both, the volume force $f = (f_1,f_2)^{\text{T}}$ and the boundary data $g = (g_1,g_2)^{\text{T}}$ are 
assumed to be 
deterministic functions. This is for the sake of simplicity of notation. Treating stochastic forcing 
and boundary data in the model would be straightforward under appropriate integrability assumptions on the data.

The remainder of the paper is organized as follows: As a model for the uncertain viscosity, we introduce the 
exponential of a Karhunen-Lo\`eve expansion (KLE) of a Gaussian random field in \cref{sec:input_model}. 
We ensure boundedness of this expansion by stating suitable assumptions on its components. The boundedness of the 
viscosity is necessary for the well-posedness of the variational formulation we introduce in \cref{sec:var_form}.
In \cref{sec:SGFE_discretization}, 
we establish a matrix representation of the Stokes problem with lognormal random data by restricting the 
weak equations to a finite dimensional subspace spanned by Taylor-Hood finite elements and Hermite chaos 
polynomials. Preconditioning strategies are discussed in \cref{sec:precon} where we consider block 
diagonal and block triangular preconditioning structures. As building blocks, we use a Kronecker product structure with 
established approaches from the FE and SG literature as input. In \cref{sec:ev_analysis}, we derive inclusion bounds for 
the eigenvalues of relevant sub-matrices and interpret them concerning the overall convergence behavior. By modifying our 
block triangular preconditioner in a specific way, we ensure the existence of a conjugate gradient (CG) method in a 
non-standard inner product. We discuss the application of this CG method as well as the application of the 
MINRES iterative solver in greater detail in \cref{sec:iterative_solvers}. A numerical test case is 
considered in \cref{sec:numerical_examples} where we illustrate the expected convergence behavior of the two considered 
solvers with respect to different problem parameters. The final section eventually summarizes and concludes our work.

\section{Input Modeling}
\label{sec:input_model}
We start our considerations with the random field $\mu({x},\stochasticDomainValue)$ and assume 
that it is Gaussian and second-order, meaning $\mu \in L^2(\stochasticDomain,L^2(\spatialDomain))$. In this 
setting it is possible to represent $\mu({x},\stochasticDomainValue)$ as a KLE of the form \cite[Theorem 5.28]{LordEtAl2014}
\begin{equation}
\label{eq:inf_KL}
\mu({x},\stochasticDomainValue) = \mu_0({x}) + \sigma_\mu \sum_{\iKL=1}^\infty \sqrt{\lambda_\iKL}
\mu_\iKL({x}) \parameter{\iKL}(\stochasticDomainValue).
\end{equation}
Here, $\mu_0({x})$ is the mean field of $\mu({x},\stochasticDomainValue)$, 
i.e.~$\mu_0({x})=\int_{\stochasticDomain}\mu({x},\stochasticDomainValue) \dd \probabilityMeasure(\stochasticDomainValue)$ 
is the expected value of the random field. The eigenpairs $(\lambda_\iKL,\mu_\iKL)_{\iKL=1}^\infty$ belong to 
the integral operator which corresponds to the covariance function of 
the correlated Gaussian random field. Further, $\parameter{\iKL},\iKL=1,\dots, \infty$, 
are uncorrelated Gaussian random variables with zero mean and unit variance which live in the unbounded 
set of sequences $\RR^\NN$. As we are in the Gaussian setting, the random variables originating from the 
KLE are also stochastically independent. 

The unbounded support of the Gaussian random variables leads to one major problem concerning the theoretical 
investigations of the problem: As there is always a nonzero probability that random variables take on 
negative values with arbitrarily large magnitudes, negative values of the modeled viscosity can occur 
independent of its construction. We use a standard approach to avoid negative values and apply the exponential 
function to \cref{eq:inf_KL}, yielding
\begin{equation}
\label{eq:inf_exp_KL}
\nu({x},\stochasticDomainValue) = \exp(\mu({x},\stochasticDomainValue)) = 
\exp \Big(\mu_0({x}) + \sigma_\mu \sum_{\iKL=1}^\infty \sqrt{\lambda_\iKL}
\mu_\iKL({x}) \parameter{\iKL}(\stochasticDomainValue) \Big),
\end{equation}
for $\parameterVector=(\parameter{\iKL}(\stochasticDomainValue))_{\iKL \in \NN} :  \stochasticDomain \to \RR^\NN$. 
Expression \cref{eq:inf_exp_KL} is called a lognormal random process as the logarithm of $\nu({x},\stochasticDomainValue)$ 
is the Gaussian process $\mu({x},\stochasticDomainValue)$. In order to ensure boundedness of the viscosity, some assumptions 
have to be made on the series components. Following \cite{HoangSchwab2014}, we 
assume that
\begin{enumerate}
 \item[$(i)$] the mean field of $\mu({x},\stochasticDomainValue)$ and the product of the Karhunen-Lo\`eve eigenpairs is bounded:
 \begin{equation}
    \label{eq:KL_assumption_i}
    \mu_0, \, \sqrt{\lambda_\iKL} \mu_\iKL \in L^\infty(\spatialDomain), \quad \forall \iKL \in \NN,
 \end{equation}
 \item[$(ii)$] the series of the product of the Karhunen-Lo\`eve eigenpairs converges absolutely:
 \begin{equation}
 \label{eq:KL_assumption_ii}
  \chi:=  \left(\chi_\iKL\right)_{\iKL\geq 1} = 
  \big(\norm{\sqrt{\lambda_\iKL}\mu_\iKL({x}) }_{L^\infty(\spatialDomain)}\big)_{\iKL\geq 1} \in \ell^1(\NN).
 \end{equation}
\end{enumerate}
Then, we define the set
\begin{equation*}
 \Xi^\nKL:= \RR^\nKL.
\end{equation*}
The set $\Xi^\nKL$ can also be defined for infinitely many parameters. However, we will not do this here as it introduces 
additional difficulties to the problem. For a full analysis, we refer to \cite{SchwabGittelson2011}.

We define the value $\overline{\mu}_0:= \norm{\mu_0({x})}_{L^\infty(\spatialDomain)}$ and the truncated viscosity:
\begin{equation}
\label{eq:inf_exp_KL_det}
\nu_\nKL({x},\parameterVector) = \exp \Big(\mu_0({x}) + \sigma_\mu \sum_{\iKL=1}^\nKL \sqrt{\lambda_\iKL}
\mu_\iKL({x}) \, y_{\iKL}\Big),
\end{equation}
for $\parameterVector \in \Xi^\nKL$. Given the assumptions \cref{eq:KL_assumption_i} and \cref{eq:KL_assumption_ii}, the viscosity \cref{eq:inf_exp_KL_det} satisfies 
\begin{align}
\label{eq:bound_visc}
 0 < \underline{\nu}(\parameterVector):= \underset{{x} \in \spatialDomain}{\text{ess inf}} \, \nu_\nKL({x},\parameterVector) 
 \leq \nu_\nKL({x},\parameterVector) \leq \underset{{x} \in \spatialDomain}{\text{ess sup}} \, \nu_\nKL({x},\parameterVector)=:
\overline{\nu}(\parameterVector),
\end{align}
with 
\begin{eqnarray*}
\overline{\nu}(\parameterVector)&\leq& \exp(\overline{\mu}_0) \exp\Big(\sum_{\iKL=1}^\nKL \chi_\iKL \abs{y_\iKL}
\Big),\\
\underline{\nu}(\parameterVector)&\geq& \exp(-\overline{\mu}_0) \exp\Big(-\sum_{\iKL=1}^\nKL \chi_\iKL \abs{y_\iKL}
\Big),
\end{eqnarray*}
for $\parameterVector \in \parameterDomain^\nKL$. A proof of \cref{eq:bound_visc} can be found in 
\cite[Lemma 2.2]{HoangSchwab2014}. 

As a consequence of \cref{eq:bound_visc}, the viscosity \cref{eq:inf_exp_KL} is bounded from above and 
has a positive lower bound for almost all $\stochasticDomainValue \in \stochasticDomain$. This is a basic 
property necessary for our problem to be well-defined. A reasonable truncation is possible with moderate $\nKL$ when 
the covariance operator of the underlying random field is sufficiently smooth and the correlation length is sufficiently large. 

The stochastic independence of the random variables allows us to express the corresponding Gaussian density as a product of univariate densities:
\begin{equation}
\label{eq:joint_density}
\probabilityDensity(\parameterVector):=\prod_{\iKL=1}^{\nKL} \rho_\iKL(y_\iKL), \quad \text{with} 
 \quad \rho_\iKL(y_\iKL)
:= \frac{1}{\sqrt{2 \pi}} \exp(-y_\iKL^2/2).
\end{equation}
It makes sense to change from the abstract the domain $\parameterDomain^\nKL$  in the following, see 
\cite[section 2.1]{HoangSchwab2014}. The change of the space leads to an integral transform, such that 
the computation of the $p$th moment of a random variable $v(\stochasticDomainValue)$ is done via
\begin{equation}
\label{eq:int_transform}
\int_\stochasticDomain v^p(\stochasticDomainValue) \dd \probabilityMeasure(\stochasticDomainValue)
\approx \int_{\Xi^\nKL} v^p(\parameterVector) \, \probabilityDensity(\parameterVector) \dd \parameterVector =: \langle v^p \rangle.
\end{equation}
As a consequence of the Doob-Dynkin lemma \cite{Oksendal1998}, the output random fields can be 
pa\-ra\-me\-tri\-zed with the vector $\parameterVector$ as well. 

\section{Variational Formulation}
\label{sec:var_form}
In order to formulate the weak equations derived from \cref{eq:continuous_stoch_problem}, we introduce 
Bochner spaces $L_\probabilityDensity^2(\Xi^\nKL;X)$, where $X$ is a separable Hilbert space. 
They consist of all equivalence classes of strongly measurable functions $v \, : \, \Xi^\nKL \to X$ with norm 
\begin{equation*}
\norm{v}_{L_\probabilityDensity^2(\Xi^\nKL;X)} = 
\left(\int_{\Xi^\nKL} \norm{v(\cdot,\parameterVector)}_X^2  \, \probabilityDensity(\parameterVector) \dd \parameterVector 
\right)^{1/2} < \infty.
\end{equation*} 
In the following, we work in the tensor product spaces $L_\probabilityDensity^2(\Xi^\nKL)\otimes X$ with corresponding 
norms $\norm{\cdot}_{L_\probabilityDensity^2(\Xi^\nKL)\otimes X} := \norm{\cdot}_{L_\probabilityDensity^2(\Xi^\nKL;X)}$ 
as they are isomorphic to the Bochner spaces for separable $X$.

For the Stokes problem with random data, we insert the standard function spaces for enclosed flow and 
introduce 
\begin{equation}
\label{eq:variational_spaces}
\begin{aligned}
\vecg{\mathcal{V}}_0 := L_{\probabilityDensity}^2(\parameterDomain^\nKL) \otimes \vecg{V}_0(\spatialDomain),  \\
\mathcal{W}_0        := L_{\probabilityDensity}^2(\parameterDomain^\nKL) \otimes W_0(\spatialDomain),
\end{aligned}
\end{equation}
with
\begin{eqnarray*}
\vecg{V}_0&:=&\vecg{H}^1_0(\spatialDomain)=\Big\{{v}\in \vecg{H}^1(\spatialDomain) \, | \, 
{v}_|{_{\spatialDomainBoundary}} = {0} \Big \}, \\
W_0 &:=&L^2_0(\spatialDomain)=\Big\{q \in L^2(\spatialDomain) \, | \, \smallint_\spatialDomain q({x}) \dd {x} = 0 \Big\}.
\end{eqnarray*}
The product spaces are Hilbert spaces as well.

We are now able to formulate the variational formulation associated with \cref{eq:continuous_stoch_problem}: 
Find $({u},p) \in \vecg{\mathcal{V}}_0 \times \mathcal{W}_0$ satisfying
\begin{equation}
\label{eq:weak_formulation}
\begin{aligned}
\langle a({u},{v}) \rangle + \langle b({v},p)\rangle &= \langle l({v})\rangle, &\quad& 
\forall {v} \in \vecg{\mathcal{V}}_0, \\
\langle b({u},q)\rangle &= \langle t(q) \rangle, &\quad& \forall q \in \mathcal{W}_0,
\end{aligned}
\end{equation}
with bilinear forms 
\begin{alignat}{3}
\label{eq:bilinear_form_a}
\langle a({u},{v})\rangle &:= \int_{\parameterDomain^\nKL}\int_{\spatialDomain}\nu_\nKL({x},\parameterVector) \, \nabla 
{u}({x},\parameterVector) \cdot \nabla {v}({x},\parameterVector) \, \probabilityDensity(\parameterVector) \dd {x} 
\dd \parameterVector, \\
\label{eq:bilinear_form_b}
\langle b({v},q) \rangle &:= - \int_{\parameterDomain^\nKL} \int_{\spatialDomain} q({x},\parameterVector) \, \nabla 
\cdot {v}({x},\parameterVector) \, \probabilityDensity(\parameterVector) \dd {x} \dd \parameterVector,
\end{alignat}
for ${u},{v} \in  \vecg{\mathcal{V}}_0$, $q \in \mathcal{W}_0$, and linear functionals 
\begin{alignat*}{3}
\langle l({v}) \rangle &:= \int_{\parameterDomain^\nKL}\int_{\spatialDomain} {f}({x}) \cdot {v}
({x},\parameterVector) \,\probabilityDensity(\parameterVector) \dd {x} \dd \parameterVector -\langle a({u}_0,{v}) \rangle , &\quad& \forall {v} \in 
\vecg{\mathcal{V}}_0,\\ 
\langle t({v}) \rangle &:= -\langle b({u}_0,q)\rangle , &\quad& \forall q \in \mathcal{W}_0,
\end{alignat*}
where ${u}_0 \in \vecg{\mathcal{V}}_0$ is the lifting of the boundary data ${g}$ in the sense of 
the trace theorem.

The weak equations \cref{eq:weak_formulation} are a set of parametric deterministic equations which contain 
the full stochastic information of the original problem with random data. For the well-posedness of the 
variational formulation \cref{eq:weak_formulation}, we refer to \cite{SchwabGittelson2011}.

\section{Stochastic Galerkin Finite Element Discretization}
\label{sec:SGFE_discretization}
We derive a discrete set of equations from \cref{eq:weak_formulation} by choosing appropriate 
subspaces for the building blocks of the product spaces \cref{eq:variational_spaces}. For the discretization 
of the physical space, we use FE subspaces $\vecg{V}_0^h \subset \vecg{V}_0$ and 
$W_0^h \subset W_0$, where $h$ denotes the mesh size. The domain of the random variables is 
discretized with generalized polynomial chaos \cite{XiuKarniadakis2002b}. The corresponding SG space is 
denoted by $S^{\stochasticDegree} \subset L_{\probabilityDensity}^2(\parameterDomain^\nKL)$, where $k$ is the 
degree of the chaos functions and $\nKL$ is the truncation index of the KLE in 
\cref{eq:inf_exp_KL_det}. The SGFE subspaces 
$\vecg{\mathcal{V}}_0^{\stochasticDegree h} \subset \vecg{\mathcal{V}}_0$ and 
$\mathcal{W}_0^{\stochasticDegree h} \subset \mathcal{W}_0$ are now defined as products of the separate 
parts:
\begin{equation}
\label{eq:variational_subspaces_parameter}
\begin{aligned}
\vecg{\mathcal{V}}_0^{\stochasticDegree h} &:= S^{\stochasticDegree} \otimes \vecg{V}_0^h,  \\
\mathcal{W}_0^{\stochasticDegree h}        &:= S^{\stochasticDegree} \otimes W_0^h.
\end{aligned}
\end{equation}
As a specific choice for the spatial discretization, we use inf-sup stable Taylor-Hood $P_2/P_1$ 
finite elements on a regular triangulation. They consist of $\nFE_u$ continuous piecewise quadratic basis 
functions for the velocity space and $\nFE_p$ continuous piecewise linear basis functions for the pressure 
space. For the parametric space, we choose a discretization based on a complete multivariate Hermite polynomial 
chaos. The corresponding basis functions are global polynomials which are orthonormal with respect to the joint 
density 
$\probabilityDensity(\parameterVector)$ in \cref{eq:joint_density}. Therefore, they are the appropriate match to the 
Gaussian distribution of the input parameters according to the Wiener-Askey scheme \cite{XiuKarniadakis2002b}. 
We construct the $\nKL$-variate basis functions as a product of $\nKL$ univariate chaos polynomials. We work 
with a complete polynomial basis, i.e.~we choose a total degree $\stochasticDegree$ and the sum of the 
degrees of the $\nKL$ univariate chaos polynomials $\sum_{\iKL=1}^{\nKL} k_\iKL$ must be less then or equal 
$\stochasticDegree$. This yields a multivariate chaos or stochastic Galerkin basis of size $\nPC_z$, where 
\begin{equation}
\label{eq:dim_sg_space}
\nPC_z:=\begin{pmatrix} \nKL + \stochasticDegree \\ \stochasticDegree \end{pmatrix}.
\end{equation}
Behind every index $\iPC=0,\dots,\nPC_z-1$ there is a unique combination of univariate polynomial degrees 
 -- a multi-index -- $(\stochasticDegree_1,\dots,\stochasticDegree_\nKL)$ and vice versa. The 
$q$th multivariate polynomial chaos basis function is thus the product of 
$\nKL$ univariate chaos polynomials with degrees from the $q$th multi-index.

Regarding the representation \cref{eq:inf_exp_KL_det}, a separation of the spatial and parametric 
dependencies would be beneficial. Such a decomposition can be achieved by representing 
the exponential of the truncated KLE in a Hermite chaos basis~\cite{PowellUllmann2010}:
\begin{equation}
\label{eq:trunc_KL_gpc}
\nu_\nKL (x,\parameterVector) = \sum_{\iPC=0}^{\nPC_\nu-1} \nu_\iPC(x) \psi_\iPC(\parameterVector),
\end{equation}
where the $\{\psi_\iPC(\parameterVector)\}_{\iPC=0}^{\nPC_\nu-1}$ are the Hermite chaos basis functions and 
\begin{equation}
\label{eq:KL_fluct}
\nu_\iPC(x) = \exp\Big(\mu_0(x) + \tfrac12 \sigma_\mu^2 \, \tsum_{\iKL=1}^\nKL \lambda_\iKL \mu^2_\iKL(x)\Big) \prod_{\iKL=1}^\nKL
\frac{\left(\sigma_\mu \sqrt{\lambda_\iKL}\mu_\iKL(x)\right)^{\stochasticDegree_\iKL}}{\sqrt{\stochasticDegree_\iKL !}}.
\end{equation}
Again, there is a unique multi-index $(\stochasticDegree_1,\dots,\stochasticDegree_\nKL)$ to every index 
$\iPC=1,\dots,\nPC_\nu$.

When used in a stochastic Galerkin setting, the representation \cref{eq:trunc_KL_gpc} is in fact exact if we use the same Hermite polynomial chaos basis 
as for the representation of the solution fields but with twice the total degree \cite[Remark 2.3.4]{EUllmann2008}, i.e.
\begin{equation}
\nPC_\nu:=\begin{pmatrix} \nKL + 2 \stochasticDegree \\ 2 \stochasticDegree \end{pmatrix}.
\end{equation}
Although $\nPC_\nu$ grows fast with $\stochasticDegree$ and $\nKL$ and can be a lot bigger than $\nPC_z$, 
it does not make sense to truncate the sum in \cref{eq:trunc_KL_gpc} prematurely. Doing so may destroy the 
coercivity of \cref{eq:bilinear_form_a} and the corresponding discrete 
operator can easily become indefinite, see \cite[Example 2.3.6]{EUllmann2008}. Consequently, we always use 
all terms in \eqref{eq:trunc_KL_gpc}. 

Without going into the details, we will assume in the following that the 
fully discrete problem is well-posed which implies that the discretizations we choose are inf-sup 
stable on the discrete product spaces \cref{eq:variational_subspaces_parameter}. An analysis of discrete 
inf-sup stability for a mixed formulation of the diffusion problem with uniform random data can be found in 
\cite[Lemma 3.1]{BespalovEtAl2012}. 

The size of the emerging system of equations for our chosen discretizations~is
\begin{equation}
\text{dim}(\vecg{\mathcal{V}}_0^{\stochasticDegree h} \times \mathcal{W}_0^{\stochasticDegree h}) =  
\nPC_z \, (\nFE_u + \nFE_p) = \nPC_z \, \nFE.
\end{equation}
To derive a matrix equation of the Stokes problem with random data, the velocity and pressure random 
fields as well as the test functions are represented in the FE and SG bases and 
subsequently inserted into the weak formulation \cref{eq:weak_formulation} together with the input 
representation \cref{eq:trunc_KL_gpc}, yielding
\begin{equation}
\label{eq:system_of_equations}
\mathcal{C} \, \vecg{w} = \vecg{b}, \quad \mathcal{C} \in \RR^{\nPC_z \nFE \times \nPC_z \nFE},
\end{equation}
where
\begin{equation}
\label{eq:block_matrix_form}
\mathcal{C}:= \begin{bmatrix} \mathcal{A} \hspace{0.2cm} & \mathcal{B}^{\text{T}} \\ \mathcal{B} \hspace{0.2cm} & 0 \end{bmatrix}, \quad 
\vecg{w}:= \begin{bmatrix} \boldsymbol{u} \\ \boldsymbol{p} \end{bmatrix}, \quad 
\vecg{b}:= \begin{bmatrix} \vecg{f} \\ \vecg{t} \end{bmatrix}.
\end{equation}
Here, the vectors $\boldsymbol{u} \in \RR^{\nPC_z \nFE_u}$ and $\boldsymbol{p} \in \RR^{\nPC_z \nFE_p}$ contain 
the coefficients of the discrete velocity and pressure solutions, respectively. Furthermore,
\begin{alignat}{2}
\label{eq:SGFEM_diffusion_matrix}
\mathcal{A} &= I \otimes A_0 + \sum_{\iPC=1}^{\nPC_\nu -1} G_{\iPC} \otimes A_{\iPC} \quad&& \in \RR^{\nPC_z  \nFE_u \times \nPC_z \nFE_u}, \\
\label{eq:SGFEM_gradient_matrix}
\mathcal{B} &= I \otimes B \quad&& \in \RR^{\nPC_z \nFE_p \times \nPC_z \nFE_u}, \\
\boldsymbol{f} &= \boldsymbol{g}_0 \otimes \boldsymbol{w} \quad&& \in \RR^{\nPC_z \nFE_u}, \\
\label{eq:div_rhs}
\boldsymbol{t} &= \boldsymbol{g}_0 \otimes \boldsymbol{d} \quad&& \in \RR^{\nPC_z \nFE_p},
\end{alignat}
where $I \in \RR^{\nPC_z \times \nPC_z}$ is the Gramian of $S^{\stochasticDegree}$, because the basis is 
orthonormal. Further, the $G_\iPC \in \RR^{\nPC_z \times \nPC_z}$ emerge from the evaluation of the  product 
of three Hermite chaos basis functions in the expectation \eqref{eq:int_transform} with $p=1$. They are called stochastic Galerkin matrices in the 
following. We call the $A_\iPC \in \RR^{\nFE_u \times \nFE_u}$ weighted FE velocity 
Laplacians as they are FE velocity Laplacians weighted with the functions 
$\nu_\iPC(x)$, $\iPC=0\,\dots,\nPC_\nu -1$. The matrices and vectors in 
\cref{eq:SGFEM_diffusion_matrix}--\cref{eq:div_rhs} can all be constructed when the FE and 
SG basis representations are inserted into the variational formulation 
\cref{eq:weak_formulation} together with the input representation \cref{eq:trunc_KL_gpc}.
To avoid confusion, the matrices on the product spaces are calligraphic capital letters whereas the matrices 
on either the FE or the SG spaces are standard capital letters.

The size of the SGFE system is the product of the size of the FE basis~$\nFE$ and 
the size of the SG basis $\nPC_z$. Additionally, the problem is coupled, as the symmetric 
matrices $G_\iPC$, $\iPC=1\,\dots,\nPC_\nu -1$, 
are not diagonal. Therefore, realistic problems are often too big to be treated with 
direct solution methods, which is why iterative algorithms are used instead. The application of iterative 
methods naturally raises the question of efficient preconditioning due to the inherent ill-conditioning of 
the problem. For the mentioned reasons, preconditioning and iterative methods are investigated in 
the following.

\section{Preconditioning}
\label{sec:precon}
The SGFE matrix $\mathcal{C} \in \RR^{\nPC_z\nFE\times\nPC_z\nFE}$ in \cref{eq:block_matrix_form} is a symmetric 
saddle point matrix. Therefore, the following factorizations exist:
\begin{equation}
\label{eq:block_factorizations}
\begin{bmatrix} \mathcal{A} \hspace{0.2cm} & \mathcal{B}^{\text{T}} \\ \mathcal{B}\hspace{0.2cm} & 0 \end{bmatrix} =
\begin{bmatrix} I & 0 \\ \mathcal{B A}^{-1} & I \end{bmatrix}
\begin{bmatrix} \mathcal{A}\hspace{0.2cm} & 0 \\ 0 \hspace{0.2cm}& \mathcal{S} \end{bmatrix}
\begin{bmatrix} I & \mathcal{A}^{-1} \mathcal{B}^{\text{T}} \\ 0 & I \end{bmatrix}
= \begin{bmatrix} \mathcal{A}\hspace{0.2cm} & 0 \\ \mathcal{B}\hspace{0.2cm}  & \mathcal{S} \end{bmatrix}
\begin{bmatrix} I & \mathcal{A}^{-1} \mathcal{B}^{\text{T}} \\ 0 & I \end{bmatrix},
\end{equation}
where $\mathcal{S}:= - \mathcal{B}\mathcal{A}^{-1}\mathcal{B}^\text{T}$ is the SGFE Schur complement.
In section \ref{sec:SGFE_discretization}, we have assumed well-posedness of the discrete variational 
problem. This implies that the discrete version of the bilinear form \cref{eq:bilinear_form_a} is continuous 
and coercive. Inserting the FE and SG basis, the coercivity condition translates into positive definiteness of the 
matrix $\mathcal{A}$. The congruence transform in \cref{eq:block_factorizations} then implies that the 
discrete SGFE problem is indefinite as the Schur complement is negative semi-definite by construction.

Motivated by the factorizations \cref{eq:block_factorizations}, we consider block diagonal and block 
triangular preconditioning structures in the following. In the context of solving saddle point problems, 
these are well established concepts \cite{BenziEtAl2005}, which are generically given by
\begin{equation}
\label{eq:precon_structure}
\mathcal{P}_1 = \begin{bmatrix} \widetilde{\mathcal{A}}\hspace{0.2cm} & 0 \\ 0\hspace{0.2cm} & 
\widetilde{\mathcal{S}}\end{bmatrix}, 
\quad
\mathcal{P}_2 = \begin{bmatrix} \widetilde{\mathcal{A}}\hspace{0.2cm} & 0 \\ 
\mathcal{B}\hspace{0.2cm} & \widetilde{\mathcal{S}}\end{bmatrix}.
\end{equation}
Here, $\widetilde{\mathcal{A}}$ and $\widetilde{\mathcal{S}}$ are approximations of $\mathcal{A}$ 
and $\mathcal{S}$, respectively. Choosing these 
approximations appropriately is the main task of preconditioning. Desirable properties 
include a reduction of computational complexity and an improvement of the condition of the involved 
operators. 

In the following, we do not directly look for suitable $\widetilde{\mathcal{A}}$ and 
$\widetilde{\mathcal{S}}$, but make another structural assumption prior to that. We want each SGFE 
preconditioner to be the Kronecker product of one FE and one SG preconditioner, 
i.e.
\begin{equation}
\label{eq:precon_ansatz_diff}
\widetilde{\mathcal{A}} := \widetilde{G}_A \otimes \widetilde{A}, \quad 
\widetilde{\mathcal{S}} := \widetilde{G}_S \otimes \widetilde{S},
\end{equation}
where $\widetilde{A} \in \RR^{\nFE_u \times \nFE_u}$ and $\widetilde{S}\in \RR^{\nFE_p \times \nFE_p}$ are 
approximations of the FE operators and the matrices $\widetilde{G}_A$, $\widetilde{G}_S \in \RR^{\nPC_z \times \nPC_z}$ are 
approximations of the SG operators. The structural simplifications \cref{eq:precon_ansatz_diff} have two advantages: 
Firstly, the Kronecker product is trivially invertible. Secondly, if we look into $\widetilde{\mathcal{A}}^{-1}\mathcal{A}$, we 
get
\begin{eqnarray}
\widetilde{\mathcal{A}}^{-1} \mathcal{A} = \bigg(\widetilde{G}_A^{-1} + \widetilde{A}^{-1}A_0 + 
\sum_{\iPC=1}^{\nPC_\nu -1} \widetilde{G}_A^{-1}  G_{\iPC} \otimes \widetilde{A}^{-1}A_{\iPC}\bigg).
\end{eqnarray}
The SG preconditioner $\widetilde{G}_A$ only acts on the SG matrices $I$ and $G_{\iPC}$ and 
the FE preconditioner $\widetilde{A}$ acts on the weighted FE Laplacians $A_{\iPC}$, $\iPC=0,\dots,\nPC_\nu-1$.
This works in the same way for the preconditioned SGFE Schur complement, as we will see in Lemma 
\ref{lemma:eigenvalues_precon_Schur} on page \pageref{lemma:eigenvalues_precon_Schur}. The separation into FE and 
SG parts thus allows the use of established preconditioners from the FE and SG literature as building blocks for 
the SGFE preconditioners. Now, we will choose suitable approximations $\widetilde{A}$ and 
$\widetilde{S}$ to the FE Laplacian and Schur complement, respectively, and suitable approximations $\widetilde{G}_A$ 
and $\widetilde{G}_S$ to the SG matrices.

\subsection{Finite Element Matrices}
\label{subsec:precon_fe_matrices}
First of all, we consider a preconditioner for the FE Laplacian and derive bounds for the eigenvalues of the 
preconditioned weighted FE Laplacians. Then, we decide on a preconditioner for the FE Schur complement.

To precondition FE Laplacians, the multigrid method has emerged as one of the most suitable 
approaches. The first reason for that is the following: one multigrid V-cycle with appropriate smoothing -- 
denoted by $\widetilde{A}_{\text{mg}}$ in the following~-- is spectrally equivalent to the FE Laplacian $A$, see 
\cite[section 2.5]{ElmanEtAl2014}. In this context, spectral equivalence means that there exist positive constants 
$\delta$ and $\Delta$, independent of $h$, such that
\begin{equation}
\label{eq:upper_bound_Diff_approxDiff}
\delta \leq \frac{\vecg{v}^{\text{T}}A\, \vecg{v}}{\vecg{v}^{\text{T}}\widetilde{A}_{\text{mg}}\vecg{v}} \leq \Delta, \quad
\forall \vecg{v} \in \RR^{\nFE_u}\backslash \{\vecg{0}\}.
\end{equation}
As the eigenvalues of $A$ actually depend on the mesh width $h$ \cite[Theorem~3.21]{ElmanEtAl2014}, 
preconditioning with $\widetilde{A}_{\text{mg}}$ thus eliminates this $h$-ill-conditioning of $A$. The multigrid method is also attractive from a 
computational point of view as it can be applied with linear complexity. Concerning \cref{eq:upper_bound_Diff_approxDiff}, 
when the multigrid operator is applied to the weighted FE Laplacians in \cref{eq:SGFEM_diffusion_matrix}, we derive 
\begin{equation}
\label{eq:bound_precon_weighed_Laplacian}
-\overline{\nu}_\iPC \, \Delta \leq \frac{\vecg{v}^{\text{T}}A_\iPC \, \vecg{v}}{\vecg{v}^{\text{T}}
\widetilde{A}_{\text{mg}}\vecg{v}} \leq \overline{\nu}_\iPC \, \Delta, \quad
\forall \vecg{v} \in \RR^{\nFE_u}\backslash \{\vecg{0}\},
\end{equation}
where $\overline{\nu}_\iPC = \norm{\nu_{\iPC}({x})}_{L^\infty(\spatialDomain)}$, for $\iPC=1,\dots,\nPC_\nu-1$. 
For $\iPC=0$, the lower bound is different, namely $\underline{\nu}_0:= \inf_{{x} \in \spatialDomain} \nu_0({x})>0$. 
This bound is tighter because we know that the function $\nu_0= \exp(\mu_0(x) + \tfrac12 \sigma_\mu^2 \, \tsum_{\iKL=1}^\nKL \lambda_\iKL \mu^2_\iKL(x))$ from \cref{eq:KL_fluct} is 
always positive. The functions are in $L^\infty(\spatialDomain)$ due to the assumptions \cref{eq:KL_assumption_i} and the 
continuity of the exponential function.

The pressure mass matrix $M_p$ is a good preconditioner for the negative FE Schur complement $-S=BA^{-1}B^{\text{T}}$, because the 
matrices are spectrally equivalent in the sense that \cite[Theorem 3.22]{ElmanEtAl2014}
\begin{equation}
\label{eq:lower_bound_Schur_Mass}
\gamma^2 \leq \frac{\vecg{q}^{\text{T}}BA^{-1}B^{\text{T}}\vecg{q}}{\vecg{q}^{\text{T}}M_p \, \vecg{q}}
\leq 1, \quad
\forall \vecg{q} \in \RR^{\nFE_p}\backslash \{\vecg{0},\vecg{1}\},
\end{equation}
where $B$ is the FE divergence matrix and $\gamma>0$ is the inf-sup constant of our mixed FE 
approximation. Further, the notation $\backslash \{\vecg{0},\vecg{1}\}$ means we exclude all multiples of the constant
function, see \cite[Section 3.3]{ElmanEtAl2014}. If necessary, we always 
exclude the hydrostatic pressure in this way in the following. As $M_p$ has the usual FE 
sparsity, using it as a preconditioner is too expensive in practice. Therefore, another approximation is considered, 
namely its diagonal $D_p := \text{diag}(M_p)$. It is spectrally equivalent 
to $M_p$, i.e.~there exist $\theta,\Theta >0$ such that
\begin{equation}
\label{eq:bound_Mass_diagMass}
\theta \leq \frac{\vecg{q}^{\text{T}}M_p\vecg{q}}{\vecg{q}^{\text{T}}D_p\vecg{q}} \leq \Theta, \quad
\forall \vecg{q} \in \RR^{\nFE_p}\backslash \{\vecg{0}\}.
\end{equation}
The constants $\theta$ and $\Theta$ only depend on the degree and type of finite elements used 
\cite{Wathen1991}. We use piecewise linear basis functions on triangles in our work, yielding 
$\theta = \frac12$ and $\Theta=2$. 
Consequently, $D_p$ is spectrally equivalent to the negative FE Schur 
complement. This directly follows from \cref{eq:lower_bound_Schur_Mass,eq:bound_Mass_diagMass}:
\begin{equation}
\label{eq:bound_Schur_diagMass}
\theta \, \gamma^2 \leq \frac{\vecg{q}^{\text{T}}BA^{-1}B^{\text{T}}\vecg{q}}{\vecg{q}^{\text{T}}
D_p\vecg{q}} 
\leq \Theta, \quad\forall \vecg{q} \in \RR^{\nFE_p}\backslash \{\vecg{0},\vecg{1}\}.
\end{equation}
Using $D_p$ as a preconditioner is also attractive from a complexity point of view as it can be 
applied with linear costs. The bounds \cref{eq:upper_bound_Diff_approxDiff}, \cref{eq:bound_precon_weighed_Laplacian} and 
\cref{eq:bound_Schur_diagMass}, which are independent of the mesh size, are the basis for analyzing the 
$h$-independence of the preconditioned SGFE saddle point problem later on in \cref{sec:ev_analysis}.

Due to the mentioned reasons, we use the multigrid V-cycle and the diagonal of the pressure mass matrix 
as FE sub-blocks in \cref{eq:precon_ansatz_diff}, fixing the choices
\begin{equation}
\widetilde{A}: = \widetilde{A}_{\text{mg}}, \quad \widetilde{S}: = D_p.
\end{equation}

\subsection{Stochastic Galerkin Matrices}
\label{subsec:precon_sg_matrices}
The SG preconditioners are also chosen according to complexity considerations and their spectral properties. 
In particular, we want to improve the condition of the SG matrices with 
respect to the discretization parameters $\stochasticDegree$ and $\nKL$ if possible. For the SG matrices 
based on a product of complete multivariate Hermite polynomials, there exist the following 
inclusion bounds \cite[Corollary 4.5]{PowellUllmann2010}:
\begin{equation}
\label{ev_bounds_G_m}
-g_\iPC \leq \frac{\vecg{a}^\text{T} G_\iPC \, \vecg{a}}{\vecg{a}^\text{T}\vecg{a}}
\leq g_\iPC, \quad \forall \vecg{a} \in \RR^{\nPC_z} \backslash \{\vecg{0}\},
\end{equation}
for $q=1,\dots,\nPC_\nu-1$, where 
$g_\iPC = \exp{(\nKL(\stochasticDegree +1)/2 +\frac12 \sum_{\iKL=1}^\nKL \stochasticDegree_\iKL)}$. As mentioned in 
section \ref{sec:SGFE_discretization} on page \pageref{eq:dim_sg_space}, the degrees $k_\iKL$ are the entries of the 
multi-index associated with the index $\iPC$. According to the bounds \cref{ev_bounds_G_m}, the eigenvalues of the 
SG matrices depend on the chaos degree $\stochasticDegree$ and the truncation index $\nKL$ of the KLE.

In the context of SGFE problems, the mean-based approximation \cite{PowellElman2009} is often 
used to construct preconditioners for SG matrices. To define an SG preconditioner, the mean information of the SGFE 
problem is used. We work with an orthonormal chaos basis and the corresponding SG matrix of the mean 
problem is the identity matrix, see \cref{eq:SGFEM_diffusion_matrix}. According to our structural 
ansatz \cref{eq:precon_ansatz_diff}, we thus define
\begin{equation}
\label{eq:mean_based_SG}
\widetilde{G}_A = \widetilde{G}_S = I. 
\end{equation}
We solely use this choice in the following although these preconditioners can not improve the condition 
of the SG matrices with respect to $\stochasticDegree$ or $\nKL$. This is because of two reasons: The mean-based 
preconditioner is extremely cheap to apply and -- to the best of our knowledge -- there is no practical preconditioner 
which can eliminate the ill-conditioning with respect to $\stochasticDegree$. There is still the potential ill-conditioning 
with respect to $\nKL$. However, our numerical experiments in \cref{sec:numerical_examples} suggest that the influence of 
$\nKL$ is not that severe.

\section{Eigenvalue Analysis for the SGFE Matrices}
\label{sec:ev_analysis}
In section \ref{sec:precon}, we fixed the structures and building blocks of our preconditioners. As a next 
step, we summarize our assumptions and choices to define the specific preconditioners we eventually use. 
Starting point for the construction of the preconditioners are the structures \cref{eq:precon_structure} and 
the substructural Kronecker product ansatz \cref{eq:precon_ansatz_diff}. 

Inserting $\widetilde{A}=\widetilde{A}_{\text{mg}}$ 
and $\widetilde{S}=D_p$ as FE preconditioners, and $\widetilde{G}_A = \widetilde{G}_S = I$ as SG preconditioners, we define:
\begin{equation}
\label{eq:block_precon}
\mathcal{P}_{\text{diag}} := \begin{bmatrix} \widetilde{\mathcal{A}}_{\text{mg}} & 0 \\
0 & \widetilde{\mathcal{S}}_p \end{bmatrix}, 
\quad
\mathcal{P}_{\text{tri}} := \begin{bmatrix} a\,\widetilde{\mathcal{A}}_{\text{mg}} & 0 \\ 
\mathcal{B} & -\widetilde{\mathcal{S}}_p \end{bmatrix},
\end{equation}
where 
\begin{align}
\label{eq:sgfe_tilde_A}
\widetilde{\mathcal{A}}_{\text{mg}}&:= I \otimes \widetilde{A}_{\text{mg}}, \\
\label{eq:sgfe_tilde_S}
\widetilde{\mathcal{S}}_p&:= I \otimes D_p.
\end{align}
Both, the negative sign of the Schur complement approximation $-\widetilde{\mathcal{S}}_p$ as well as the 
scalar $a \in \RR$ in the definition of the block triangular preconditioner $\mathcal{P}_{\text{tri}}$ are 
manipulations which are not necessary in general. However, they are essential for one of the iterative solvers 
we are using. The specific effects of these additional manipulations are specified in subsection \ref{subsec:block_tri}.

In order to asses the influence of different problem parameters on the spectrum of the preconditioned SGFE 
systems, we want to derive inclusion bounds for the eigenvalues of $\mathcal{P}_{\text{diag}}^{-1} \mathcal{C}$ 
and $\mathcal{P}_{\text{tri}}^{-1} \mathcal{C}$ defined in \cref{eq:block_matrix_form} and 
\cref{eq:block_precon}. We can do this using existing results from saddle point theory. For the block 
diagonal preconditioned problem, we use 
\cite[Theorem 3.2, Corollary 3.3 and Corollary 3.4]{PowellSilvester2003}. These results imply that the 
eigenvalues of $\mathcal{P}_{\text{diag}}^{-1} \mathcal{C}$ depend on the eigenvalues of 
 -- in our notation -- $\widetilde{\mathcal{A}}_{\text{mg}}^{-1}\mathcal{A}$ and 
$\widetilde{\mathcal{S}}_p^{-1}\mathcal{B}\widetilde{\mathcal{A}}_{\text{mg}}^{-1}\mathcal{B}^\text{T}$. 
However, in order to apply these results, both preconditioned sub-matrices must be positive definite.

To derive estimates on the eigenvalues of the block triangular preconditioned matrix, we want to apply 
\cite[Theorem 4.1]{Zulehner2001}. This result bounds the eigenvalues of 
$\mathcal{P}_{\text{tri}}^{-1} \mathcal{C}$ by the eigenvalues of the preconditioned SGFE Laplacian and the 
preconditioned SGFE Schur complement. In our setting, it can be applied if 
\begin{eqnarray}
\label{eq:block_tri_cond_i}
a\,\widetilde{\mathcal{A}}_{\text{mg}} &<& \mathcal{A} \leq \alpha_2 \widetilde{\mathcal{A}}_{\text{mg}}, \\
\label{eq:block_tri_cond_ii}
\hat{\gamma} \widetilde{\mathcal{S}}_p &\leq& \mathcal{B}\mathcal{A}^{-1}\mathcal{B}^\text{T} \leq
\hat{\Gamma} \widetilde{\mathcal{S}}_p,
\end{eqnarray}
where $a$ is the scaling introduced in \cref{eq:block_precon} and $\alpha_2$, $\hat{\gamma}$ and $\hat{\Gamma}$ 
are positive constants.
\subsection{The Block Diagonal Preconditioned SGFE System}
\label{subsec:ev_blk_diag}
First of all, we consider the preconditioned SGFE Laplacian. To derive bounds on the eigenvalues of 
$\widetilde{\mathcal{A}}_{\text{mg}}^{-1}\mathcal{A}$, we proceed as in \cite[Lemma~7.2]{MuellerEtAl2017}, 
\cite[Theorem~4.6]{PowellUllmann2010}.
\begin{lemma}
\label{lemma:eigenvalues_precon_diff}
Let the matrices $\mathcal{A}$ and $\widetilde{\mathcal{A}}_{\text{mg}}$ be defined according to
\cref{eq:SGFEM_diffusion_matrix} and \cref{eq:sgfe_tilde_A}, respectively. Then,
\begin{align}
\label{eq:ev_bounds_SGFE_diff_approx_diff}
\hat{\delta} \leq \frac{\vecg{v}^{\text{T}}\mathcal{A} \, \vecg{v}}
{\vecg{v}^{\text{T}} \widetilde{\mathcal{A}}_{\text{mg}} \, \vecg{v}} \leq \hat{\Delta}, \quad 
\forall \vecg{v} \in \RR^{\nPC \nFE_u} \backslash \{\vecg{0}\},
\end{align}
with
\begin{align} 
\label{eq:ev_bounds_SGFE_delta_hat}
\hat{\delta} &:= \underline{\delta}  > 0, \\
\label{eq:ev_bounds_SGFE_Delta_hat}
\hat{\Delta} &:= (\overline{\nu}_0 + \nu_\sigma) \Delta,
\end{align}
where $\underline{\delta}$ is a positive constant not further specified and 
$\nu_\sigma:=\sum_{\iPC=1}^{\nPC_\nu -1} g_\iPC \, \overline{\nu}_\iPC$.
\end{lemma}
\begin{proof}
First of all, we bound the eigenvalues from above:
\begin{align}
\label{eq:bounds_SGFE_diff_upper}
\lambda_{\text{max}}\left(\widetilde{\mathcal{A}}_{\text{mg}}^{-1} \, \mathcal{A} \right) &= 
\max_{\vecg{v}\in\RR^{\nPC \nFE_u} \backslash \{\vecg{0}\}}
\frac{\vecg{v}^{\text{T}}\left(I \otimes A_0 + \sum_{\iPC=1}^{\nPC_z-1}G_\iPC \otimes A_\iPC \right) \vecg{v}}
{\vecg{v}^{\text{T}}\left(I \otimes \widetilde{A}_{\text{mg}} \right) \vecg{v}} \\
\notag
&\leq \lambda_{\text{max}}\left(I \otimes \widetilde{A}_{\text{mg}}^{-1}A_0 \right) + 
\sum_{\iPC=1}^{\nPC_z -1}
\lambda_{\text{max}}\left( G_{\iPC} \otimes \widetilde{A}_{\text{mg}}^{-1}A_\iPC \right) \\
&\overset{\cref{eq:bound_precon_weighed_Laplacian},\cref{ev_bounds_G_m}}{\leq} \overline{\nu}_0 \, \Delta + 
\sum_{\iPC=1}^{\nPC_\nu -1} g_\iPC \, \overline{\nu}_\iPC \, \Delta = (\overline{\nu}_0 +  
\nu_\sigma) \Delta.
\end{align}
An analog procedure for the lower bound on the eigenvalues yields
\begin{equation}
\label{eq:bounds_SGFE_diff_lower}
\lambda_{\text{min}}\left(\widetilde{\mathcal{A}}_{\text{mg}}^{-1} \, \mathcal{A} \right) \geq 
\underline{\nu}_0 \, \delta - \nu_\sigma \, \Delta.
\end{equation}
However, due to the rough bounds $g_\iPC$ entering $\nu_\sigma$, expression \cref{eq:bounds_SGFE_diff_lower} is 
likely to be negative. This does not contradict the theory, but the results for the block diagonal preconditioned saddle point 
problem do not hold for a negative lower bound. The discrete well-posedness assumption ensures the existence of 
a positive lower bound $\lambda_{\text{min}}\left( \mathcal{A} \right) \geq \underline{\alpha} > 0$ 
as it implies discrete coercivity. As $\widetilde{\mathcal{A}}_{\text{mg}}^{-1}$ is positive definite as well, 
there is also a positive constant $\underline{\delta}$ fulfilling
\begin{equation}
\label{eq:bounds_SGFE_diff_lower_positive}
\lambda_{\text{min}}\left(\widetilde{\mathcal{A}}_{\text{mg}}^{-1} \, \mathcal{A} \right) \geq 
\underline{\delta} > 0,
\end{equation}
such that the result can be applied. However, we do not have any further information on $\underline{\delta}$, especially 
not on the parameter dependencies hidden in the bound.
\vspace{-0.65cm}
\flushright $\blacksquare$
\end{proof}
Bounds on the eigenvalues of 
$\widetilde{\mathcal{S}}_p^{-1}\mathcal{B}\widetilde{\mathcal{A}}_{\text{mg}}^{-1}\mathcal{B}^\text{T}$ can 
be derived as in \cite[Lemma 7.3]{MuellerEtAl2017}:
\begin{lemma}\cite[Lemma 7.3]{MuellerEtAl2017}
\label{lemma:eigenvalues_precon_quasi_Schur}
Let $\widetilde{\mathcal{A}}_{\text{mg}}$ and $\widetilde{S}_p$ be defined as in 
\cref{eq:sgfe_tilde_A} and \cref{eq:sgfe_tilde_S}. Then
\begin{equation}
\label{eq:bounds_SGFE_quasi_Schur_approx_Schur}
\delta \, \theta \, \gamma^2 \leq \frac{\vecg{q}^{\text{T}} \mathcal{B}\widetilde{\mathcal{A}}_{\text{mg}}^{-1}
\mathcal{B}^{\text{T}}\vecg{q}}
{\vecg{q}^{\text{T}}\widetilde{S}_p \vecg{q}} 
\leq \Delta \, \Theta, \quad \forall \vecg{q} \in \RR^{\nPC \nFE_p} \backslash \{\vecg{0},\vecg{1}\},
\end{equation}
\end{lemma}
\begin{proof}
See proof of \cite[Lemma 7.3]{MuellerEtAl2017}. 
\vspace{-0.5cm}
\flushright $\blacksquare$
\end{proof}
The constants in the bounds in \eqref{eq:bounds_SGFE_quasi_Schur_approx_Schur} only depend on the degree and 
type of finite elements we use and on the shape of the considered domain. They do not depend on discretization or 
modeling parameters.

Combining Lemma \ref{lemma:eigenvalues_precon_diff} and Lemma \ref{lemma:eigenvalues_precon_quasi_Schur} with the 
results \cite[Theorem 3.2, Corollary 3.3 and Corollary 3.4]{PowellSilvester2003}, we find that the eigenvalues of 
$\mathcal{P}_{\text{diag}}^{-1} \mathcal{C}$ are bounded by a combination of \cref{eq:ev_bounds_SGFE_delta_hat}, 
\cref{eq:ev_bounds_SGFE_Delta_hat} and the bounds in \cref{eq:bounds_SGFE_quasi_Schur_approx_Schur}. The bounds in 
\cref{eq:bounds_SGFE_quasi_Schur_approx_Schur} are parameter independent, so we look into 
\cref{eq:ev_bounds_SGFE_delta_hat} and \cref{eq:ev_bounds_SGFE_Delta_hat}. The bound \cref{eq:ev_bounds_SGFE_Delta_hat} 
suggests that the eigenvalues change with the functions $\nu_\iPC$, $\iPC=0,\dots,\nPC_z-1$ and the parameters contained in 
the SG matrix bound $g_\iPC$. Looking at the definition \eqref{ev_bounds_G_m}, the eigenvalues might thus be influenced by the chaos 
degree $k$ and the KLE truncation index $\nKL$. Lastly, we consider \cref{eq:ev_bounds_SGFE_delta_hat}:
As we do not have information on the parameter dependencies hidden in $\underline{\delta}$, any other problem parameter -- 
such as the mesh width $h$ -- can potentially influence the eigenvalues.

\subsection{The Block Triangular Preconditioned SGFE System}
\label{subsec:block_tri}
In order to bound the eigenvalues of $\mathcal{P}_{\text{tri}}^{-1} \mathcal{C}$ using \cite[Theorem 4.1]{Zulehner2001},
we need to fulfill \eqref{eq:block_tri_cond_i} and \eqref{eq:block_tri_cond_ii}. The right bound of 
\eqref{eq:ev_bounds_SGFE_diff_approx_diff} is an upper bound fulfilling \eqref{eq:block_tri_cond_i} but the left bound of 
\eqref{eq:ev_bounds_SGFE_diff_approx_diff} is not necessarily a lower bound fulfilling \eqref{eq:block_tri_cond_i}. 
The lower bound in \eqref{eq:block_tri_cond_i} is in fact the stronger condition:
\begin{equation}
\frac{\vecg{v}^{\text{T}}\mathcal{A} \, \vecg{v}} {\vecg{v}^{\text{T}} a \widetilde{\mathcal{A}}_{\text{mg}} 
\, \vecg{v}} > 1 , \quad 
\forall \vecg{v} \in \RR^{\nPC \nFE_u} \backslash \{\vecg{0}\}.
\end{equation}
or equivalently 
\begin{equation}
\label{eq:condition_pd_bpcg}
\lambda_{\text{min}}\left((a \widetilde{\mathcal{A}}_{\text{mg}})^{-1} \mathcal{A}\right) = 
\min_{\vecg{v}\in \RR^{\nPC \nFE_u} \backslash \{\vecg{0}\}} 
\frac{\vecg{v}^{\text{T}}\mathcal{A} \vecg{v}}{\vecg{v}^{\text{T}} a \widetilde{\mathcal{A}}_{\text{mg}}
\vecg{v}} > 1.
\end{equation}
Now, the scalar $a$ comes into play. If it is chosen such that 
$a = \kappa \, \lambda_{\text{min}}(\widetilde{\mathcal{A}}_{\text{mg}}^{-1} \mathcal{A})$ with 
$ 0< \kappa < 1$, then a scaled preconditioner 
$\kappa \, \lambda_{\text{min}}(\widetilde{\mathcal{A}}_{\text{mg}}^{-1} \mathcal{A}) \, \widetilde{\mathcal{A}}_{\text{mg}}$
fulfills \cref{eq:condition_pd_bpcg} with 
\begin{equation}
\lambda_{\text{min}}\left((\kappa \, \lambda_{\text{min}}(\widetilde{\mathcal{A}}_{\text{mg}}^{-1} \mathcal{A}) \, \widetilde{\mathcal{A}}_{\text{mg}})^{-1} \mathcal{A}\right) = 
\kappa^{-1} > 1.
\end{equation}
However, as we have no quantitative access to the analytical lower bound \cref{eq:ev_bounds_SGFE_delta_hat}, we need 
to estimate the minimum eigenvalue of $\widetilde{\mathcal{A}}_{\text{mg}}^{-1} \mathcal{A}$ numerically. A 
scaled preconditioner with a numerically computed positive 
$a=a^* < \lambda_{\text{min}}(\widetilde{\mathcal{A}}_{\text{mg}}^{-1} \mathcal{A})$ then yields a modified 
version of \cref{eq:ev_bounds_SGFE_diff_approx_diff}, namely
\begin{align}
\label{eq:ev_bounds_SGFE_diff_approx_diff_scaled}
1 < \frac{\hat{\delta}}{a^*} \leq \frac{\vecg{v}^{\text{T}}\mathcal{A} \, \vecg{v}}
{\vecg{v}^{\text{T}} a^* \, \widetilde{\mathcal{A}}_{\text{mg}} \, \vecg{v}} \leq \frac{\hat{\Delta}}{a^*} , \quad 
\forall \vecg{v} \in \RR^{\nPC \nFE_u} \backslash \{\vecg{0}\},
\end{align}
which fulfills \eqref{eq:block_tri_cond_i}.

As we have assumed well-posedness of the discrete SGFE problem, a discrete inf-sup condition is fulfilled 
\cite[Theorem 3.18]{John2016}. Using the discrete representations of the norms and inner 
products in the FE and SG basis, we can rearrange the discrete inf-sup condition to derive a relation exactly as in 
\cite[Lemma 3.48 and section 3.6.6]{John2016}. In our SGFE setting, the spectral equivalence has the form 
\begin{equation}
\label{eq:spec_equiv_Schur}
\beta^2 \leq \frac{\vecg{q}^{\text{T}} (I \otimes B A^{-1} B^{\text{T}}) \vecg{q}}{\vecg{q}^{\text{T}} 
(I \otimes M_p) \vecg{q}} \leq 1, \quad \forall \vecg{q} \in \RR^{\nPC \nFE_p}\backslash\{\vecg{0},\vecg{1}\}.
\end{equation}
Here, $\beta$ is the inf-sup constant of the mixed SGFE problem. We use \cref{eq:spec_equiv_Schur} to derive 
bounds for the preconditioned SGFE Schur complement in the following lemma.
\begin{lemma}
\label{lemma:eigenvalues_precon_Schur}
Let the Schur complement be defined by $\mathcal{S} = -\mathcal{B}  \mathcal{A}^{-1} \mathcal{B}^{\text{T}}$
with building blocks $\mathcal{A}$ and $\mathcal{B}$ in \cref{eq:SGFEM_diffusion_matrix} and 
\cref{eq:SGFEM_gradient_matrix}, and $\widetilde{\mathcal{S}}_p = I \otimes D_p$ according to \cref{eq:sgfe_tilde_S}. 
Then,
\begin{equation}
\label{eq:ev_bounds_SGFE_Schur_approx_Schur}
\hat{\gamma} \leq \frac{\vecg{q}^\text{T} \, \mathcal{B}  \mathcal{A}^{-1} \mathcal{B}^{\text{T}} \, \vecg{q}}
{\vecg{q}^\text{T} \left( I \otimes D_p \right) \vecg{q}} \leq \hat{\Gamma},
\quad \vecg{q} \in\RR^{\nPC \nFE_p}\backslash\{\vecg{0},\vecg{1}\},
\end{equation}
with
\begin{align} 
\label{eq:ev_bounds_SGFE_gamma_hat}
\hat{\gamma} &:= (\underline{\nu}_0 + \nu_\sigma)^{-1} \theta \, 
\beta^2 , \\
\label{eq:ev_bounds_SGFE_Gamma_hat}
\hat{\Gamma} &:= \underline{\delta}^{-1} \, \Delta \, \Theta.
\end{align}
\end{lemma}
\begin{proof}
We want to modify \cref{eq:spec_equiv_Schur} in order to derive \cref{eq:ev_bounds_SGFE_Schur_approx_Schur}. 
The denominators are matched directly via \cref{eq:bound_Mass_diagMass} as the additional identity matrices do not 
change the bounds. The connection between the nominators is not that obvious. We start by considering
\begin{equation}
\label{eq:ev_bounds_SGFE_Schur_2} 
(I \otimes A)^{-1} \mathcal{A}
= I \otimes A^{-1}A_0 + \sum_{\iPC=1}^{\nPC_\nu -1} G_\iPC \otimes A^{-1}A_\iPC.
\end{equation}
Extracting the weighting factors from the FE Laplacians yields
\begin{equation}
\label{eq:ev_bounds_SGFE_Schur_4} 
- \overline{\nu}_\iPC \leq 
\frac{\vecg{v}^{\text{T}}A_\iPC  \vecg{v}}
{\vecg{v}^{\text{T}} A \,\vecg{v}} \leq \overline{\nu}_\iPC,
\end{equation}
for all $\iPC=1,\dots\nPC_\nu -1$ and for all $\vecg{v} \in \RR^{\nFE_u} \backslash \{\vecg{0}\}$. By 
combining \cref{eq:ev_bounds_SGFE_Schur_4} with the representation 
\cref{eq:ev_bounds_SGFE_Schur_2} and \cref{ev_bounds_G_m}, we get
\begin{equation}
\label{eq:ev_bounds_SGFE_Schur_5} 
\frac{\vecg{v}^{\text{T}}(I\otimes A)^{-1} \vecg{v}}
{\vecg{v}^{\text{T}} \mathcal{A}^{-1} \vecg{v}} \leq \overline{\nu}_0 + 
\sum_{\iPC=1}^{\nPC_\nu-1} g_\iPC \, \overline{\nu}_\iPC = \overline{\nu}_0 + \nu_\sigma,
\end{equation}
for all $\vecg{v} \in \RR^{\nPC \nFE_u} \backslash \{\vecg{0}\}$. For the lower bound, we start with 
\cref{eq:ev_bounds_SGFE_diff_approx_diff}: 
\begin{align}
\label{eq:ev_bounds_SGFE_Schur_6} 
 \frac{\vecg{v}^{\text{T}}\mathcal{A} \, \vecg{v}}
{\vecg{v}^{\text{T}} \widetilde{\mathcal{A}}_{\text{mg}} \, \vecg{v}} \geq \underline{\delta} , \quad 
\forall \vecg{v} \in \RR^{\nPC \nFE_u} \backslash \{\vecg{0}\}.
\end{align}
As $\widetilde{\mathcal{A}}_{\text{mg}} = I \otimes \widetilde{A}_{\text{mg}}$, we can use the inverse of 
\cref{eq:upper_bound_Diff_approxDiff} with \cref{eq:ev_bounds_SGFE_Schur_6} and derive
\begin{align}
\label{eq:ev_bounds_SGFE_Schur_7} 
\frac{\vecg{v}^{\text{T}}(I\otimes A)^{-1} \vecg{v}}
{\vecg{v}^{\text{T}} \mathcal{A}^{-1} \vecg{v}} \geq \underline{\delta} \, \Delta^{-1} , \quad 
\forall \vecg{v} \in \RR^{\nPC \nFE_u} \backslash \{\vecg{0}\}.
\end{align}
Now, we use the relation $\vecg{v} = \mathcal{B}^{\text{T}}\vecg{q}$ in 
\cref{eq:ev_bounds_SGFE_Schur_5} and \cref{eq:ev_bounds_SGFE_Schur_6} to establish:
\begin{equation}
\label{eq:ev_bounds_SGFE_Schur_8} 
(\underline{\nu}_0 + 
\nu_\sigma)^{-1}
\leq \frac{\vecg{q}^{\text{T}} \mathcal{B}\mathcal{A}^{-1} \mathcal{B}^{\text{T}}\vecg{q}}
{\vecg{q}^{\text{T}} \mathcal{B} (I\otimes A)^{-1} \mathcal{B}^{\text{T}}\vecg{q}} \leq
\underline{\delta}^{-1} \, \Delta, 
\end{equation}
for all $\vecg{q} \in \RR^{\nPC \nFE_p} \backslash \{\vecg{0},\vecg{1}\}$. 
Multiplying the positive expressions 
\cref{eq:spec_equiv_Schur} and \cref{eq:ev_bounds_SGFE_Schur_8} and using \cref{eq:bound_Mass_diagMass} 
then yields the assertion. 
\vspace{-0.5cm}
\flushright $\blacksquare$
\end{proof}
Using \eqref{eq:ev_bounds_SGFE_diff_approx_diff_scaled} and Lemma \ref{lemma:eigenvalues_precon_Schur} with 
the result \cite[Theorem 4.1]{Zulehner2001} implies, that the eigenvalues of $\mathcal{P}_{\text{tri}}^{-1} \mathcal{C}$ 
with $a=a^* < \lambda_{\text{min}}(\widetilde{\mathcal{A}}_{\text{mg}}^{-1} \mathcal{A})$ 
can be bounded by a combination of the bounds in \cref{eq:ev_bounds_SGFE_diff_approx_diff_scaled}, 
\cref{eq:ev_bounds_SGFE_gamma_hat} and \cref{eq:ev_bounds_SGFE_Gamma_hat}. The upper bound in 
\cref{eq:ev_bounds_SGFE_diff_approx_diff_scaled} and \cref{eq:ev_bounds_SGFE_gamma_hat} suggest that the eigenvalues are influenced 
by the scaling $a^*$, the chaos degree $k$ and the KLE truncation index 
$\nKL$ hidden in the bounds $g_\iPC$, and the functions~$\nu_\iPC$. However, we do 
not know which parameters are hidden in \cref{eq:ev_bounds_SGFE_delta_hat}
and \cref{eq:ev_bounds_SGFE_Gamma_hat}. Therefore, we can not exclude the possibility that the eigenvalues 
change with other problem parameters such as the mesh width~$h$.

We now introduce the matrix
\begin{align}
\label{eq:inner_product_matrix}
\mathcal{H} := \begin{bmatrix} \mathcal{A} - a \widetilde{\mathcal{A}}_{\text{mg}} & 0 \\ 0 & 
\widetilde{\mathcal{S}}_p \end{bmatrix},
\end{align}
with $a = \kappa \, \lambda_{\text{min}}(\widetilde{\mathcal{A}}_{\text{mg}}^{-1} \mathcal{A})$, $0<\kappa<1$, 
$\widetilde{\mathcal{A}}_{\text{mg}}$ and $\widetilde{\mathcal{S}}_p$ according to \cref{eq:sgfe_tilde_A} 
and \cref{eq:sgfe_tilde_S}. We need this matrix in the following lemma, which can be proven 
because of the modifications to the block triangular preconditioner $\mathcal{P}_{\text{tri}}$ in 
\cref{eq:block_precon}.
\begin{lemma}
\label{lemma:ipcg_conditions}
Let $a$ in \cref{eq:block_precon} be set to 
$a = \kappa \, \lambda_{\text{min}}(\widetilde{\mathcal{A}}_{\text{mg}}^{-1} \mathcal{A})$, with 
$0 < \kappa< 1$. Then, the matrix $\mathcal{H}$ in
\cref{eq:inner_product_matrix} defines an inner product and the triangular preconditioned system 
matrix $\mathcal{P}_{\text{tri}}^{-1} \mathcal{C}$ with $\mathcal{C}$ and $\mathcal{P}_{\text{tri}}$ according to 
\cref{eq:block_matrix_form} and \cref{eq:block_precon}, is $\mathcal{H}$-symmetric and 
$\mathcal{H}$-positive definite, i.e.
\begin{align}
\label{eq:inner_product_cond_2}
\hspace{2cm} &\mathcal{H} \mathcal{P}_{\text{tri}}^{-1} \mathcal{C} = (\mathcal{P}_{\text{tri}}^{-1}\mathcal{C})^T\mathcal{H}, \\
\label{eq:inner_product_cond_1}
&\vecg{z}^{\text{T}} \mathcal{H} \mathcal{P}_{\text{tri}}^{-1} \mathcal{C} \vecg{z} > 0, \hspace{2cm} \forall 
\vecg{z} \in \RR^{\nPC \nFE} \backslash\{\vecg{0}\}.
\end{align}
\end{lemma}
\begin{proof}
See proof of \cite[Lemma 8.1]{MuellerEtAl2017}.
\vspace{-0.4cm}
\flushright $\blacksquare$
\end{proof}
Bramble and Pasciak discovered the extraordinary effect of the matrix $\mathcal{H}$ in their 1988 paper 
\cite{BramblePasciak1988}. They considered a triangular preconditioned problem structurally equivalent 
to $\mathcal{P}_{\text{tri}}^{-1} \mathcal{C}$. Due to the conditions verified in Lemma \ref{lemma:ipcg_conditions}, 
a conjugate gradient method exists in the $\mathcal{H} \mathcal{P}_{\text{tri}}^{-1} \mathcal{C}$-inner product 
\cite{FaberManteuffel1984}. 
The Bramble-Pasciak conjugate gradient (BPCG) method introduced in \cite{BramblePasciak1988} can thus also be 
applied to the SGFE problem considered in this work.

\section{Iterative Solvers}
\label{sec:iterative_solvers}
In the following, we discuss two different iterative solvers for our SGFE Stokes problem. Similar 
investigations were conducted in \cite{PetersEtAl2005} for Stokes flow with deterministic data and in 
\cite{MuellerEtAl2017} for Stokes flow with uniform random data.

As the system matrix $\mathcal{C}$ of the SGFE Stokes problem \cref{eq:system_of_equations} is 
symmetric but indefinite (see \cref{eq:block_factorizations}), the MINRES method is the first iterative 
solver one usually considers. 
It is attractive from a complexity point of view as the Krylov basis 
is built with short recurrence. However, MINRES relies on the symmetry of the problem and can thus only be 
combined with a symmetric preconditioner. Applying e.g. the nonsymmetric $\mathcal{P}_{\text{tri}}$ 
in \cref{eq:block_precon} may prevent MINRES from converging. Concerning the preliminary considerations in the 
sections \ref{sec:precon} and \ref{sec:ev_analysis}, we will thus use MINRES solely in combination with 
$\mathcal{P}_{\text{diag}}$ in \cref{eq:block_precon}.
The MINRES algorithm can be formulated efficiently such that the only matrix-vector operations 
necessary per iteration are the application of $\mathcal{C}$ and $\mathcal{P}_{\text{diag}}^{-1}$, see 
\cite[Algorithm~4.1]{ElmanEtAl2014}. Bounds on the extreme eigenvalues of $\mathcal{P}_{\text{diag}}^{-1}\mathcal{C}$ 
are often used to asses MINRES convergence behavior a priori. This is done because of the standard convergence result 
\cite[Theorem 4.14]{ElmanEtAl2014} which bounds MINRES iteration counts by a function of those eigenvalues. Following this reasoning, we will 
use the results from subsection \ref{subsec:ev_blk_diag} to interpret the numerical MINRES results in section \ref{sec:numerical_examples}.

The SGFE problem is no longer symmetric when $\mathcal{P}_{\text{tri}}^{-1}$ in \cref{eq:block_precon} is applied 
to $\mathcal{C}$. Faber and Manteuffel \cite{FaberManteuffel1984} proofed that there does not exist a short recurrence for generating an 
orthogonal Krylov subspace basis for every nonsymmetric matrix. However, they showed that there are special 
cases for which it is possible. The combination of $\mathcal{P}_{\text{tri}}^{-1} \mathcal{C}$ and 
$\mathcal{H}$ is such a case, as condition \cref{eq:inner_product_cond_2} holds. Systems of equations associated 
with the nonsymmetric matrix $\mathcal{P}_{\text{tri}}^{-1} \mathcal{C}$ can thus 
be solved with a CG method. Besides this big advantage, there are also drawbacks: firstly, for the 
method to be defined properly, we must scale $\widetilde{\mathcal{A}}_{\text{mg}}$ 
such that \cref{eq:condition_pd_bpcg} holds. This can be achieved by choosing the scaling as 
$a = \kappa \, \lambda_{\text{min}}(\widetilde{\mathcal{A}}_{\text{mg}}^{-1} \mathcal{A})$, $0 <\kappa <1$, 
as discussed in section \ref{subsec:block_tri}. However, solving the associated eigenproblem numerically leads 
to additional computational costs. Secondly, the naive BPCG algorithm is associated with the 
$\mathcal{H} \mathcal{P}_{\text{tri}}^{-1} \mathcal{C}$-inner product \cite[section 4]{AshbyEtAl1990}. Evaluating 
quantities in this inner product would lead to additional matrix-vector operations. 
Due to certain properties of CG methods \cite{AshbyEtAl1990}, these additional 
costs can be avoided by reformulating the algorithm. Thereby, a BPCG algorithm can be found 
which needs only one extra operation compared to preconditioned MINRES \cite[section 3.1]{PetersEtAl2005}: a 
matrix-vector multiplication with $\mathcal{B}$. This additional operation originates from the definition of 
$\mathcal{P}_{\text{tri}}$ in \cref{eq:block_precon} and is cheap compared to a multiplication with 
$\mathcal{A}$, because $\mathcal{B}$ is block diagonal with sparse blocks. For this reason, the BPCG method is 
particularly interesting in our setting where $\mathcal{A}$ is block dense 
\cite[Lemma 28]{ErnstUllmann2010}. There is also a convergence result for CG which bounds its iteration counts 
by a function of the extreme eigenvalues of $\mathcal{P}_{\text{tri}}^{-1} \mathcal{C}$ 
\cite[Theorem 9.4.12]{Hackbusch1994}, so we follow the same arguments as above: We use the results from 
subsection \ref{subsec:block_tri} to interpret the numerical CG results in \cref{sec:numerical_examples}.

\section{Numerical Experiments}
\label{sec:numerical_examples}
In the following, we compare the solvers discussed in \cref{sec:iterative_solvers} numerically. We use 
the regularized driven cavity test case and investigate the performance of the two methods as well as their 
convergence behavior.

We associate the random field $\mu({x},\stochasticDomainValue)$ with the separable exponential 
covariance function $C_\mu \, : \, \spatialDomain \times \spatialDomain \to \RR$:
\begin{equation}
\label{eq:covariance}
C_\mu(x,y) = \sigma^2_\mu \, e^{-\abs{x_1-y_1}/b_1 - \abs{x_2-y_2}/b_2},
\end{equation}
with correlation lengths $b_1$ and $b_2$ in the $x_1$ and $x_2$ direction, respectively. 
Eigenpairs of the two-dimensional integral operator associated with \cref{eq:covariance} 
are constructed by combining the eigenpairs of two one-dimensional operators, which can be calculated 
analytically \cite[section 5.3]{GhanemSpanos1991}, \cite[section 7.1]{LordEtAl2014}.

We use a regularized version of the lid-driven cavity \cite[section 3.1]{ElmanEtAl2014} as our 
test case. The spatial domain is the unit square $\spatialDomain=[-0.5,0.5]\times[-0.5,0.5]$ and 
we impose a parabolic flow profile ${u}({x}) = (1-16x_1^4,0)^T$ at the top lid. No-slip conditions are 
enforced everywhere else on the boundary. For the numerical simulations, we use the following default 
parameter set:
\begin{equation}
\label{eq:default_parameter_set}
h=0.01, \quad \stochasticDegree=1, \quad \nKL= 10, \quad \nu_0 = 1, \quad \sigma_\nu = 0.2, \quad b_1=b_2=1.
\end{equation}
If not specified otherwise, the simulation parameters are the ones in \cref{eq:default_parameter_set}. The mean 
and variance of the corresponding velocity streamline field can be found in Figure \ref{figure:example_solution}.
\begin{figure}[tbhp]
\centering
\includegraphics[width=0.48\textwidth]{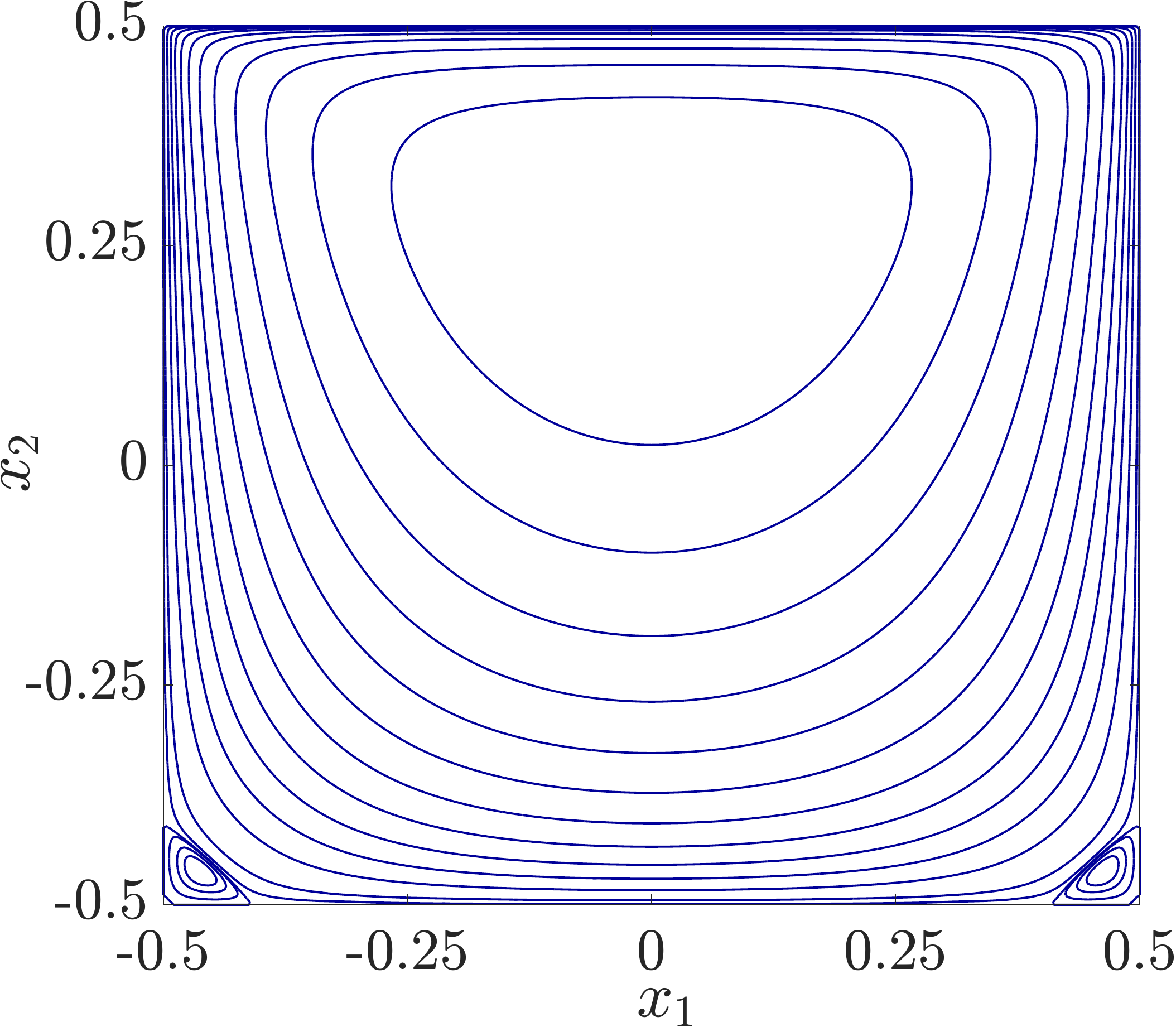} \quad
\includegraphics[width=0.48\textwidth]{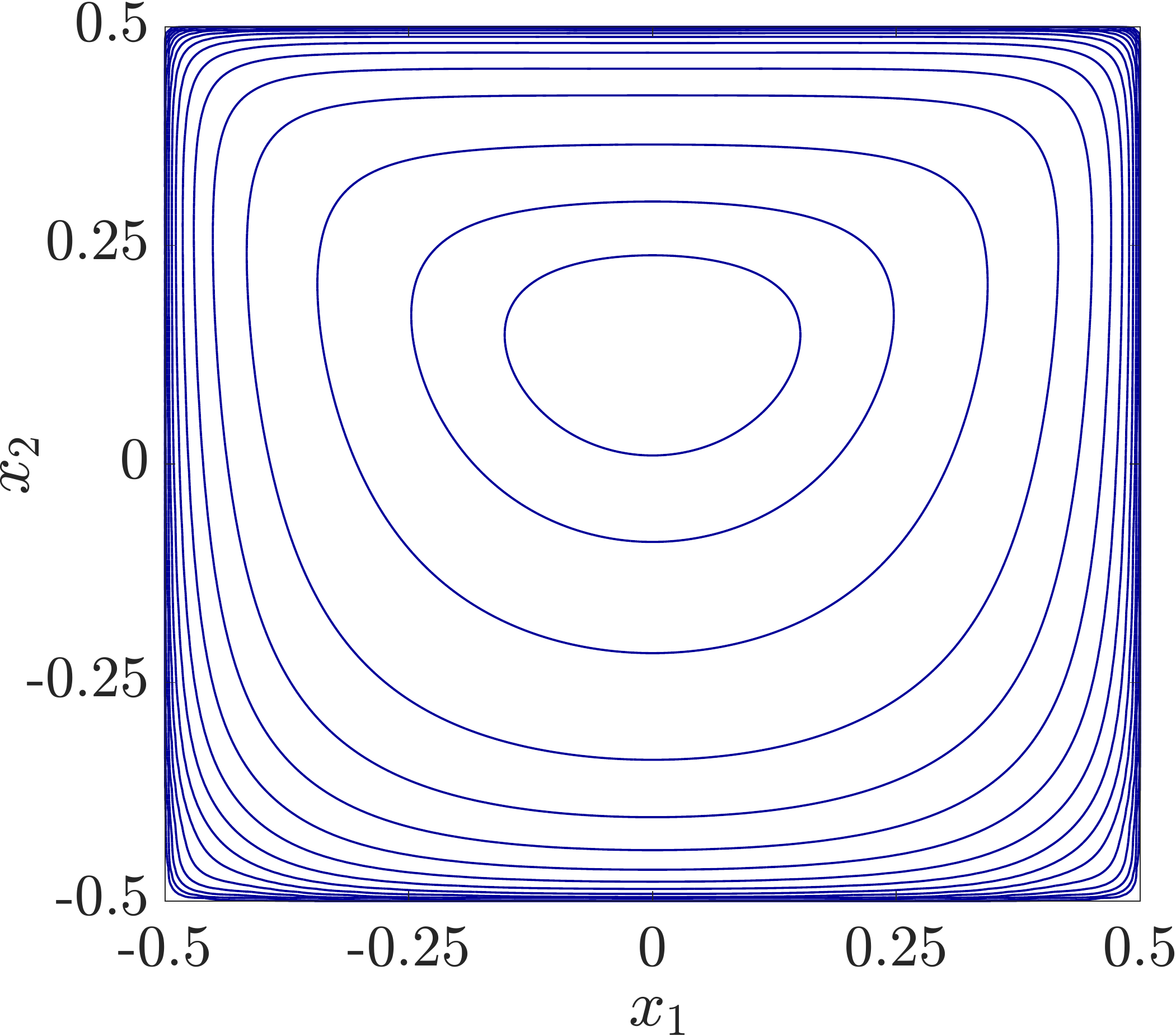}
\caption{Contour lines of the mean (left) and variance (right) of the stream function of the regularized 
driven cavity test case computed with the parameters in \cref{eq:default_parameter_set}.}
\label{figure:example_solution}
\end{figure}

All numerical simulations are carried out in our own finite element implementation in MATLAB 
\cite{Ullmann2016}, except the setup of the multigrid preconditioner. For this particular issue, we resort 
to the algebraic multigrid implementation in the IFISS package \cite{ElmanEtAl2014b} with two point 
Gauss-Seidel pre-and post-smoothing sweeps. In order to compare 
$\mathcal{P}_{\text{diag}}$-pre\-con\-di\-tio\-ned MINRES ($\mathcal{P}_{\text{diag}}$-MINRES) and 
$\mathcal{P}_{\text{tri}}$-preconditioned BPCG ($\mathcal{P}_{\text{tri}}$-BPCG), we look at the iteration 
counts necessary to reduce the Euclidean norm of the relative residual below 
$10^{-6}$. As the initial guess is always the zero vector, we thus consider numbers $n$ such that
\begin{equation}
\norm{\vecg{r}^{(n)}} = \norm{\vecg{b}-\mathcal{C}\vecg{z}^{(n)}} \leq 10^{-6} \norm{\vecg{b}}.
\end{equation}
Before we commence with the actual comparison, we want to assess the influence of the scaling factor $a$ 
on the $\mathcal{P}_{\text{tri}}$-BPCG convergence. Based on the reference value 
$a^* \approx \lambda_{\text{min}}(\widetilde{\mathcal{A}}_{\text{mg}}^{-1}\mathcal{A})$ that we compute 
with MATLAB's (version 8.6.0) numerical eigensolver \texttt{eigs}, we solve the driven cavity problem for 
different values of the relative scaling $a/a^*$. Corresponding iteration counts for 
$\mathcal{P}_{\text{tri}}$-BPCG are displayed in Table \ref{table:varying_scaling}. The minimum iteration 
count of 32 
is attained when the scaling $a$ is chosen to be the reference value $a^*$. Consequently, the ideal 
scaling of the preconditioner is close to the border of the $\mathcal{H}$-positive definiteness condition~\cref{eq:condition_pd_bpcg}. 
However, we notice that moderate variations around the optimal scaling do not lead to a significant increase in 
iteration counts. Further, we can not guarantee convergence for the algorithm when 
$a/a^*>1$ as the $\mathcal{H}$-positive definiteness requirement \cref{eq:condition_pd_bpcg} is no longer strictly 
fulfilled. Still, we could not observe divergent behavior in our experiments.

To lower the costs associated with computing $a^*$, we solve the associated 
eigenproblems on the coarsest mesh with $h=0.1$. This is somewhat heuristic as we could not show $h$-independence 
of the bounds in Lemma \ref{lemma:eigenvalues_precon_diff}. Nevertheless, we assume that the 
mesh size does not influence the scaling significantly due to the chosen spectrally equivalent FE preconditioners, 
see subsection \ref{subsec:precon_fe_matrices}.
\begin{table}[tbhp]
\caption{$\mathcal{P}_{\text{tri}}$-BPCG iteration counts for different values of the relative scaling $a/a^*$.}
\label{table:varying_scaling}
\centering
\begin{tabular}{l c c c c c c c c c c c c c}
\toprule
  $a/a^*$ \hspace{0.5cm}& 0.1 \hspace{0.5cm}& 0.4 \hspace{0.5cm}& 0.6 \hspace{0.5cm}& 0.8 \hspace{0.5cm}& 0.9 \hspace{0.5cm}& 1.0 \hspace{0.5cm}& 1.1 \hspace{0.5cm}& 1.2 \hspace{0.5cm}& 1.4 \hspace{0.5cm}& 1.6 \hspace{0.5cm}& 2.0 \hspace{0.5cm}& 3.0  \hspace{0.5cm}& 5.0 \\
  \midrule
  $n$     & 46  & 41  & 39  & 35  & 34  & 32  & 35  & 35  & 35  & 36  & 38  & 41   & 42 \\ 
\bottomrule
\end{tabular}
\end{table}

We now look at the iteration counts of the two solvers when different parameters 
are varied, starting with the mesh size $h$ and the truncation index $\nKL$ of the KLE. The associated 
numerical results are displayed in Figure \ref{figure:variation_h_M}. 
First of all, we observe that $\mathcal{P}_{\text{tri}}$-BPCG converges in fewer iterations than $\mathcal{P}_{\text{diag}}$-MINRES 
for all considered values of $h$ and $\nKL$. Further, the iteration counts of both solvers do not increase 
under mesh refinement but rather decrease slightly, as can be seen in the left plot. In the right
plot, we notice that the iteration counts increase up to $\nKL \approx 5$ for both methods and then 
basically stay constant independent of~$\nKL$. The results in Figure \ref{figure:variation_h_M} suggest 
that the iteration counts are asymptotically independent of the mesh size and the KLE truncation index. 
This is according to expectations for $h$, as the multigrid V-cycle and the 
diagonal of the pressure mass matrix are spectrally equivalent to the weighted FE Laplacians and 
Schur complement, see \eqref{eq:upper_bound_Diff_approxDiff} -- \eqref{eq:bound_Schur_diagMass}. Asymptotic 
independence of~$\nKL$ is somewhat surprising, as this 
parameter appears in \cref{eq:ev_bounds_SGFE_Delta_hat}, hidden in $\nu_\sigma$. This dependence 
originates from the bounds on the SG matrices in \cref{ev_bounds_G_m}. 

Figure \ref{figure:variation_k_s} visualizes the convergence behavior of the two considered solvers when either the 
total degree $\stochasticDegree$ of the chaos basis or the standard deviation $\sigma_\mu$ of the original Gaussian 
process $\mu({x},\stochasticDomainValue)$ in \cref{eq:inf_KL} is varied. 
We again observe that $\mathcal{P}_{\text{diag}}$-MINRES consistently needs more iterations to converge than 
$\mathcal{P}_{\text{tri}}$-BPCG. We can further see a steady increase of iteration counts with both $\stochasticDegree$ and 
$\sigma_\mu$ for both solvers. That was to be expected as these parameters also occur in the bounds 
\cref{eq:ev_bounds_SGFE_Delta_hat} and \cref{eq:ev_bounds_SGFE_gamma_hat}. When they increase, the fluctuation 
parts -- i.e.~the terms in the sum in \cref{eq:SGFEM_diffusion_matrix} -- become more important, see 
\cref{eq:KL_fluct} and \cref{ev_bounds_G_m}. 

In order to alleviate the influence on $\stochasticDegree$ and $\sigma_\mu$, one needs to use more advanced 
approaches for the SG preconditioners such as the Kronecker product preconditioner, 
see \cite{PowellUllmann2010,EUllmann2010}. However, as there is -- to the best of our knowledge~-- no practical 
preconditioner that can eliminate just one of these dependencies and using a more elaborate preconditioner also 
results in increased computational costs, we do not investigate this issue further here.

\begin{figure}[tbhp]
\centering
\includegraphics[width=0.48\textwidth]{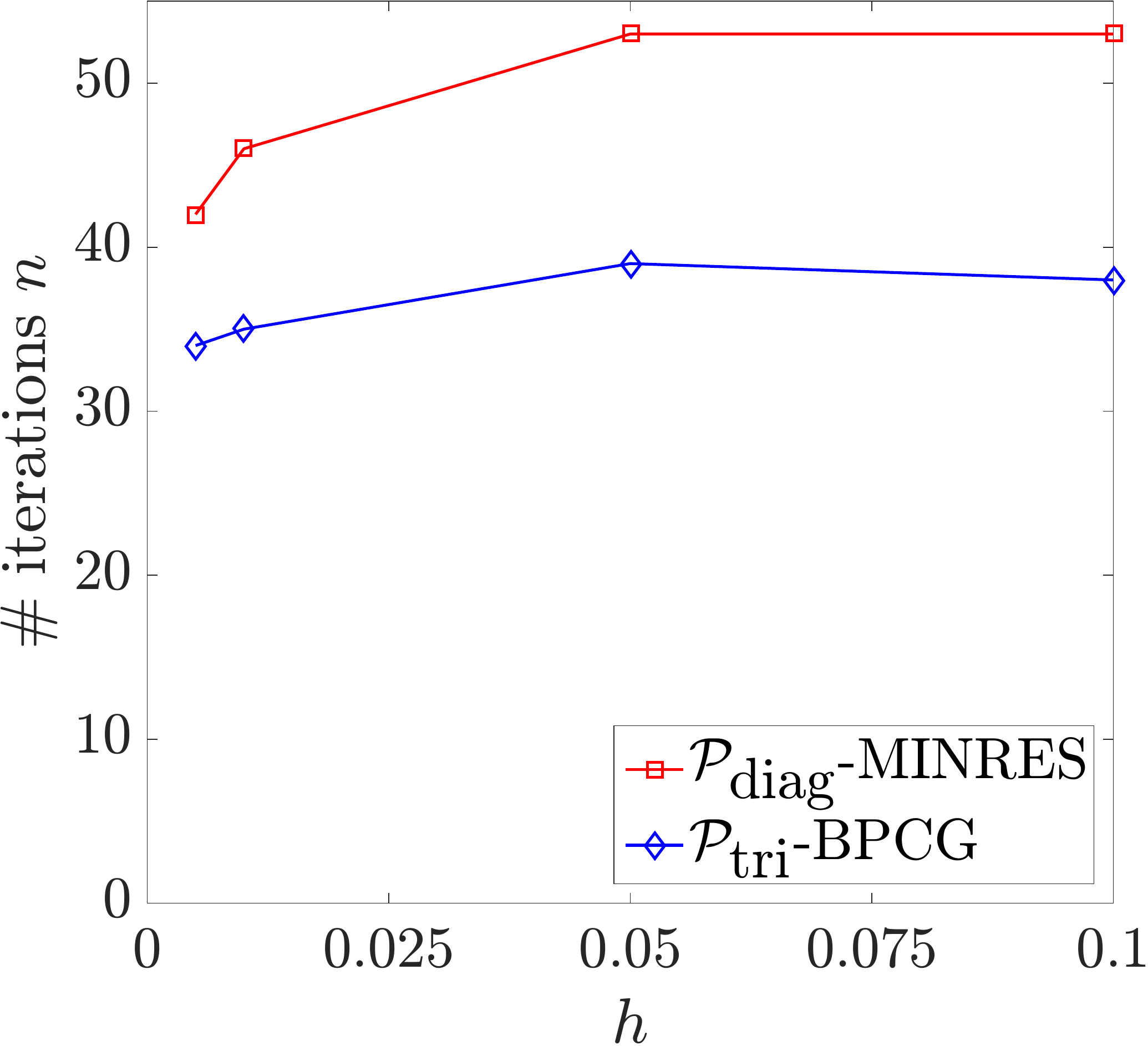} \quad
\includegraphics[width=0.48\textwidth]{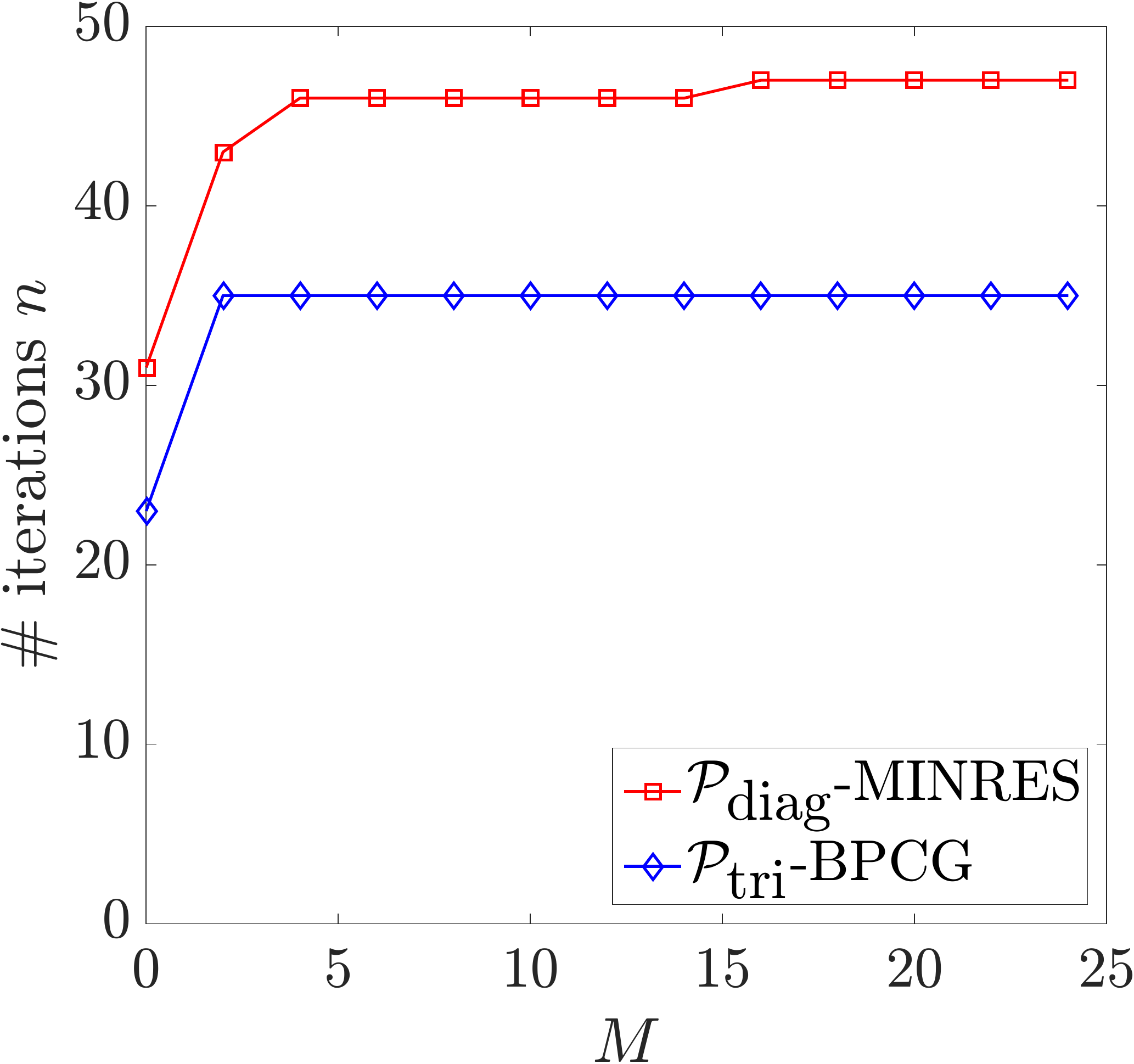}
\caption{Iteration counts for different values of the mesh size $h$ (left) and the 
KLE truncation index $\nKL$ (right) for $\mathcal{P}_{\text{diag}}$-MINRES (red) and $\mathcal{P}_{\text{tri}}$-BPCG 
(blue).}
\label{figure:variation_h_M}
\end{figure}

\begin{figure}[tbhp]
\centering
\includegraphics[width=0.48\textwidth]{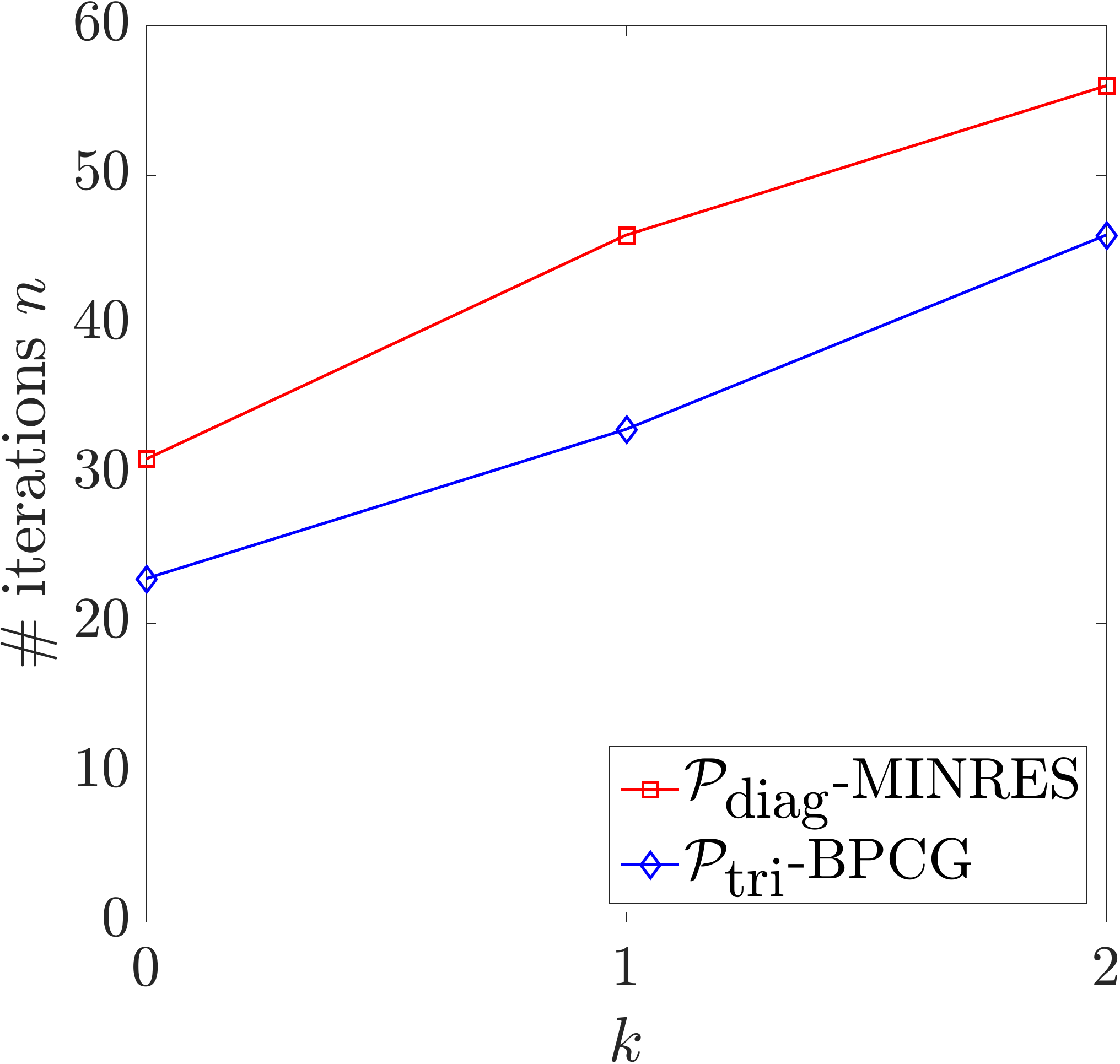} \quad
\includegraphics[width=0.48\textwidth]{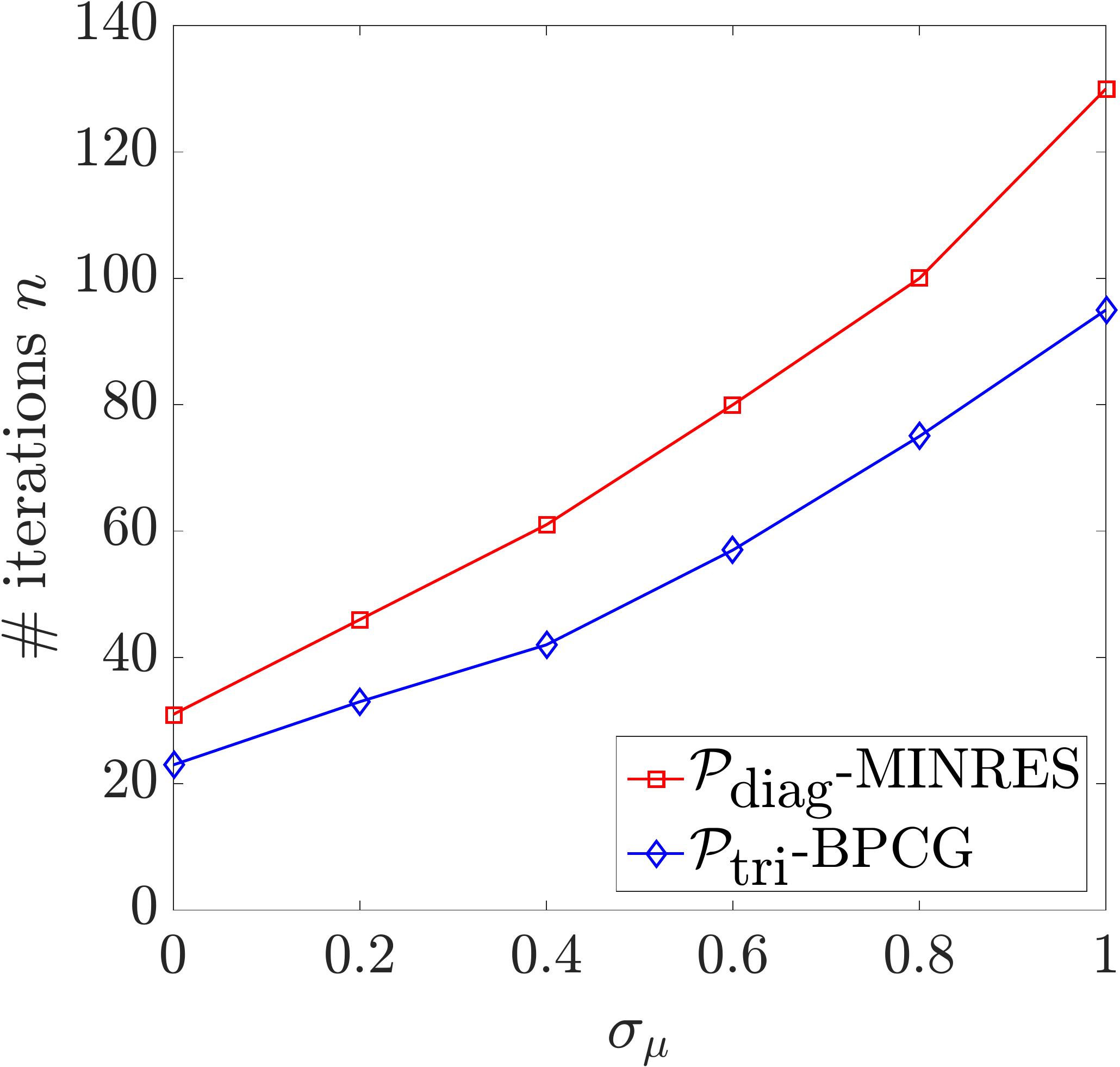}
\caption{Iteration counts for different values of the chaos degree $\stochasticDegree$ (left) and the standard 
deviation $\sigma_\nu$ (right) of the Gaussian field $\mu(x,\stochasticDomainValue)$ for 
$\mathcal{P}_{\text{diag}}$-MINRES (red) and $\mathcal{P}_2$-BPCG (blue).}
\label{figure:variation_k_s}
\end{figure}

\section{Conclusion}
The construction of our BPCG solver relies on the appropriate choice of the scaling $a$ such that the matrix 
$\mathcal{H}$ is positive definite. Choosing $a$ close to the minimum eigenvalue of the preconditioned SGFE 
Laplacian is optimal, as confirmed by the numerical experiments. However, solving the associated eigenproblem 
numerically is often prohibitive. To reduce the costs of computing $a$, one can solve the eigenproblem on a coarser 
mesh. This approach worked well in the numerical examples we considered.

We compared the iteration counts of two iterative solvers with structurally different preconditioners: block 
diagonal preconditioned MINRES and block triangular preconditioned BPCG. The iteration counts of the latter 
were consistently lower in our experiments. However, as we used the same FE and SG building blocks, 
the performance was qualitatively the same: The application of the multigrid preconditioner and the diagonal of the 
pressure mass matrix resulted in iteration counts basically independent of 
the mesh width $h$. The input dimension $\nKL$ is a measure for the accuracy of the input representation. 
However, as soon as a certain threshold is reached, the iteration counts stayed constant independent of $\nKL$. This suggests 
that the eigenvalues of the SG matrices are asymptotically independent of $\nKL$, an assertion which is not according to the 
current theory. Both the degree of the polynomial chaos $\stochasticDegree$ as well as the standard deviation $\sigma_\mu$ 
critically influence the condition of our preconditioned problems. This is already visible in the available eigenvalue 
bounds. The mean-based SG preconditioner can not alleviate these influences. Therefore, when we increased one of these 
parameters, iteration counts increased as well.

Summarizing the investigations of the BPCG method with block triangular preconditioner and the MINRES 
method with block diagonal preconditioner, we can state the following: The eigenvalue analysis is largely inconclusive 
mainly due to the coarse inclusion bounds for the eigenvalues of the SG matrices. However, our numerical tests suggest 
that the methods perform similarly to each other and essential behave as expected. If the scaling parameter can be obtained cheaply, 
the application of the block triangular preconditioner can result in a noticeable reduction of iteration 
counts compared to the application of the block diagonal preconditioner. This can lead to a reduction 
of the overall computational costs, because one step of block diagonal preconditioned MINRES is -- especially in the SGFE case -- 
only marginally cheaper than one step of block triangular preconditioned BPCG. \\
\phantom{bla}\\
\textbf{Acknowledgment. } This work is
supported by the Excellence Initiative of the German federal and state
governments and the Graduate School of Computational Engineering at 
Technische Universit\"at Darmstadt.

%

\ifx\justbeingincluded\undefined
\end{document}
\fi